\documentclass[10pt]{article}
\usepackage{amsmath}
\usepackage{amssymb}
\usepackage{enumerate}
\usepackage{graphicx}
\usepackage{pifont}
\usepackage{epsfig}
\usepackage{theorem}
\usepackage{amsmath}
\usepackage{amssymb}
\usepackage{enumerate}
\usepackage{graphicx}
\usepackage{pifont}
\usepackage{epsfig}
\usepackage{color}

\newcommand{\lev}[1]{\ensuremath{\mathrm{lev}_{\leq #1}\:}}

\newcommand{\menge}[2]{\big\{{#1}~\big |~{#2}\big\}} 
\newcommand{\emp}{\ensuremath{\varnothing}}
\newcommand{\Frac}[2]{\displaystyle{\frac{#1}{#2}}} 
\newcommand{\scal}[2]{\big\langle{{#1},{#2}}\big\rangle} 
\newcommand{\HH}{\ensuremath{{\mathcal H}}}
\newcommand{\GG}{\ensuremath{\mathcal G}}

\newcommand{\RR}{\ensuremath{\mathbb R}}
\newcommand{\RP}{\ensuremath{\left[0,+\infty\right[}}

\newcommand{\RPP}{\ensuremath{\,\left]0,+\infty\right[}}
\newcommand{\RX}{\ensuremath{\,\left]-\infty,+\infty\right]}}

\newcommand{\NN}{\ensuremath{\mathbb N}}
\newcommand{\LL}{\ensuremath{\mathbb L}}
\newcommand{\KK}{\ensuremath{\mathbb K}}
\newcommand{\dom}{\ensuremath{\mathrm{dom}\,}}
\newcommand{\rprox}{\ensuremath{\mathrm{rprox}}}
\newcommand{\prox}{\ensuremath{\mathrm{prox}}}

\newcommand{\inte}{\ensuremath{\mathrm{int}\,}}
\newcommand{\rint}{\ensuremath{\mathrm{rint}\,}}
\newcommand{\ran}{\ensuremath{\mathrm{ran}\,}}

\newcommand{\Id}{\ensuremath{\mathrm{Id}}}
\newcommand{\ic}{\ensuremath{\mathrm{\iota_C}}}

\newcommand{\minf}{\ensuremath{-\infty}}
\newcommand{\pinf}{\ensuremath{+\infty}}

\newcommand{\zz}{\ensuremath{\tilde{z}}}

\newcommand{\gammat}{\ensuremath{\kappa}}
\newcommand{\lambdat}{\ensuremath{\tau}}
\newcommand{\yy}{\ensuremath{y}}
\newcommand{\yyy}{\ensuremath{y'}}

\newcommand \suite[2][m]{\left(#2\right)_{#1\in \NN}}

\newcommand \norm[1]{\Vert#1\Vert}
\newcommand \normCar[1]{\ensuremath{\norm{#1}^{2}}}
\newcommand \prodScal[2]{\left\langle #1,#2 \right\rangle}

\newtheorem{theorem}{Theorem}[section]
\newtheorem{lemma}[theorem]{Lemma}
\newtheorem{example}[theorem]{Example}
\newtheorem{problem}{Problem}[section]
\newtheorem{proposition}[theorem]{Proposition}
\newtheorem{assumption}[theorem]{Assumption}
\newtheorem{algorithm}{Algorithm}[section]
\newtheorem{remark}[theorem]{Remark}

\title{Nested iterative algorithms for convex constrained image recovery problems\thanks{Part of this work appeared in the conference proceedings of EUSIPCO 2008 \cite{Pustelnik_N_2008_peusipco_constrained_fbairp}. This work was supported by the Agence Nationale de la Recherche under grant ANR-05-MMSA-0014-01.}}

\author{Caroline Chaux \and Jean-Christophe Pesquet  \and Nelly Pustelnik 
\thanks{C. Chaux, J.-C. Pesquet and N. Pustelnik are with the Universit{\'e} Paris-Est, Institut Gaspard Monge and CNRS-UMR 8049, 77454 Marne-la-Vall{\'e}e
Cedex 2, France. Phone: +33 1 60 95 77 39, E-mail: \texttt{\{caroline.chaux,jean-christophe.pesquet,nelly.pustelnik\}@univ-paris-est.fr}.}}
\begin{document}

\maketitle

\begin{abstract}
The objective of this paper is to develop methods for solving image recovery problems subject to constraints on the solution. More precisely, we will be interested in problems
which can be formulated as the minimization over a closed convex constraint set of the sum of two convex functions $f$ and $g$, where $f$ may be non-smooth and $g$ is differentiable 
with a Lipschitz-continuous gradient.
To reach this goal, we derive two types of algorithms that combine forward-backward and  Douglas-Rachford iterations. The weak convergence of the proposed algorithms is proved. In the case when the Lipschitz-continuity 
property of the gradient of $g$ is not satisfied, we also show that, under some assumptions,
it remains possible to apply these methods to the considered optimization problem by making use of a quadratic extension technique.
The effectiveness of the algorithms is demonstrated for two wavelet-based image restoration problems involving a signal-dependent Gaussian noise and a Poisson noise, respectively.

\end{abstract}

\markboth{Pustelnik \emph{et al.}: A constrained forward-backward algorithm\ldots}{SIAM Journal on Imaging Sciences}

\section{Introduction}\label{sec:intro}

Wavelet decompositions \cite{Mallat_S_1999_ap_wavelet_awtosp} proved their efficiency in solving many inverse problems. More recently, frame representations such as Bandlets \cite{LePennec_E_2005_tip_spa_girb}, Curvelets \cite{Candes_EJ_2002_as_Recovering_eiipipoocf}, Grouplets \cite{Mallat_S_2008_acha_geometric_g} or dual-trees \cite{Selesnick_I_2005_dsp_dual_tdtcwt,Chaux_C_2006_tip_ima_adtmbwt} have gained much popularity. These linear tools provide geometrical representations of images and they are able to easily incorporate a priori information (e.g. via some simple statistical models) on the data. Variational or Bayesian formulations of inverse problems using such representations
often lead to the minimization of convex objective functions including a non-differentiable term having a sparsity promoting role \cite{Chambolle_A_1998_tip_nonlin_wipvpcnrws,Nikolova_M_2000_siam_Local_shoare,Antoniadis_A_2002_sn_Wavelet_tfscongn,Candes_EJ_2006_ip_Sparsity_aiics,Tropp_JA_2006_tit_convex_jrcpmissn,Combettes_PL_2007_siamopt_proximal_tafmoob}.

In restoration problems, the observed data are corrupted by a linear operator and a noise which is not necessarily additive. To solve this problem, one can adopt a variational approach, aiming at minimizing the sum of two functions $f$ and $g$ over a convex set $C$ in the transform domain.
Throughout the paper, $f$ and $g$ are assumed to be in the class $\Gamma_0(\HH)$ of lower semicontinuous convex functions 
taking their values in $]-\infty,+\infty]$ which are proper (i.e. not identically equal to $+\infty$) and defined on a real separable Hilbert space $\HH$.
Then, our objective is to solve the following:
\begin{problem}\label{pb:minimisation}
Let $C$ be a nonempty closed convex subset of $\HH$. Let $f$ and $g$ be in $\Gamma_{0}(\HH)$, where $g$ is differentiable on $\HH$ with a $\beta$-Lipschitz continuous gradient for some $\beta\in\RPP$.
\begin{equation*}
\text{Find}\qquad \min_{x\in C} f(x)+g(x).
\end{equation*}
\end{problem}
Problem \ref{pb:minimisation} is equivalent to minimizing $f+g+\ic$, where $\ic$ denotes the indicator function of $C$, i.e. 
\begin{equation*}
\label{e:iota}
(\forall x\in\HH)\quad\ic(x)=
\begin{cases}
0,&\text{if}\;\;x\in C;\\
\pinf,&\text{otherwise.}
\end{cases}
\end{equation*}
Up to now, many authors devoted their works to the unconstrained case, i.e. $C= \HH$.
So-called  thresholded Landweber algorithms belonging to the more general class
of forward-backward optimization methods were proposed in \cite{Figueiredo_M_2003_tosp_EM_afwbir,Bect_J_2004_eccv_unified_vfir,Daubechies_I_2004_cpamath_iterative_talipsc,Bredies_K_2007_coa_gene_cgmcism} in order to solve the problem numerically.
Daubechies \textit{et al.} \cite{Daubechies_I_2004_cpamath_iterative_talipsc} investigated the convergence of these algorithms in the particular case when $g$ is a quadratic function and $f$ is a weighted $\ell_p$-norm with $p\in\left[1,2\right]$. These approaches were put into a more general convex analysis framework in \cite{Combettes_PL_2005_mms_Signal_rbpfbs} and extended to frame representations in \cite{Chaux_C_2007_ip_variational_ffbip}.
Attention was also paid to the improvement of the convergence speed of the forward-backward algorithm in \cite{Bioucas_J_2007_toip_New_ttsistafir},
for some specific choices of $f$ and $g$.
In \cite{Vonesch_C_2008_tip_Fast_lafwrmd}, an accelerated method was suggested in the specific case of a deconvolution in a Shannon wavelet basis.
Then, a Douglas-Rachford algorithm  relaxing the assumption of differentiability of $g$
was introduced in \cite{Combettes_PL_2007_istsp_Douglas_rsatncvsr}. In recent works \cite{Dupe_FX_2008_pisbi_deconv_cmipisr,Dupe_FX_2008_ip_proximal_ifdpniusr}, a variational approach, which is grounded on a judicious use of the Anscombe transform, was developed for the deconvolution of data contaminated by Poisson noise. A modification of the forward-backward algorithm
was subsequently proposed in finite dimension in order to solve the associated optimization problem. Additional comments concerning this approach will be given in Sections \ref{sec:icfDR} and \ref{sec:poisson}.
A key tool in the study of the aforementioned methods is the proximity operator introduced by Moreau in 1962 \cite{Moreau_JJ_1962_cras_Fonctions_cdeppdueh,Moreau_J_1965_bsmf_Proximite_eddueh}. The proximity operator of  $f \in  \Gamma_0(\HH)$ is 
$\prox_f\colon\HH \to \HH\colon x \mapsto \arg\min_{y \in \HH} \Frac12\left\|y-x\right\|^{2} + f(y)$. We thus see that $\prox_{\ic}$ reduces to the projection $P_C$ onto the convex set $C$. The function $f$ in Problem~\ref{pb:minimisation} may be non-smooth and, actually, it is often chosen as an $\ell^1$-norm, in which case its proximity operator reduces to a componentwise soft-thresholding \cite{Combettes_PL_2005_mms_Signal_rbpfbs}.
In \cite{Combettes_PL_2007_siamopt_proximal_tafmoob}, the authors derived the concept of proximal thresholding by considering a larger set of non-differentiable convex functions.

The goal of this paper is to propose iterative algorithms 
allowing us to solve Problem~\ref{pb:minimisation} when $C\neq \HH$.
The relevance of the proposed methods is shown for 
image recovery problems where convex constraints on the solution need to be satisfied.

In Section \ref{sec:tools}, we start by recalling some properties of the proximity operator. Then, in Section \ref{se:compprox} we briefly describe the forward-backward and Douglas-Rachford methods. As the proximity operator of the sum of the indicator
function of a convex set and a function in $\Gamma_0(\HH)$ cannot be easily expressed
in general, we propose two iterative methods to compute this operator: the first one is a forward-backward algorithm, whereas the second one is a Douglas-Rachford algorithm.
We also investigate the specific convergence properties of these two algorithms.
In Section \ref{sec:algorithm}, we derive two iterative methods to solve Problem \ref{pb:minimisation} and their convergence behaviours are studied. Finally, in Section \ref{sec:application}, these algorithms are applied to a class of image recovery problems. In this case, the Lipschitz-continuity property of the gradient of $g$ is not satisfied in the considered maximum a posteriori criterion. To overcome this difficulty, a quadratic extension technique providing a lower approximation of the objective function is introduced. Numerical results concerning 
deconvolution problems in the presence of signal-dependent Gaussian noise or Poisson noise
are then provided.

\section{Some properties of proximity operators}\label{sec:tools}
As already mentioned, the proximity operator of $\ic + f$ plays a key role in our approach.
Some  useful results for the calculation of  $\prox_{\ic + f}$ are first recalled.
Subsequently, the domain of a function $f\,:\,\HH \to ]-\infty,+\infty]$ is denoted by $\dom f = \{x \in \HH \;\mid\; f(x) < +\infty\}$.
\begin{proposition}{\rm \cite[Proposition 12]{Combettes_PL_2007_istsp_Douglas_rsatncvsr}} \label{prop:jstsp}
Let $f\in \Gamma_0(\mathcal{H})$ and let $C$ be a closed convex subset of $\mathcal{H}$ such that $C\,\cap\, \dom f\, \neq \, \emp$. Then the following properties hold.
\begin{enumerate}
\item\label{prop:jstspi}  $(\forall x\in \mathcal{H}),\, \prox_f x\in C\Rightarrow \prox_{\ic+f}x = \prox_{f}x$
\item\label{prop:jstspii} Suppose that $\mathcal{H}=\mathbb{R}$. Then 
\begin{equation}
\prox_{\ic+f} = P_C\circ \prox_f. \label{eq:ppcp}
\end{equation}
\end{enumerate}
\end{proposition}

Note that, the second part of this proposition can be generalized, yielding the following result which appears also as an extension of \cite[Proposition~2.10]{Chaux_C_2007_ip_variational_ffbip} when $C\neq\HH$:
\begin{proposition}
\label{l:decomp}
Let $\KK$ be a nonempty subset of \NN, $(o_k)_{k\in\KK}$ be an orthonormal basis of $\HH$ and $(\varphi_k)_{k\in\KK}$ be functions in $\Gamma_0(\RR)$.
Set
\begin{equation}
f\colon\HH\to\RX\colon x\mapsto
\sum_{k\in\KK}\varphi_k(\scal{x}{o_k}).\label{eq:fsep}
\end{equation}
Let
\begin{equation}
C = \bigcap_{k\in \KK} \{x\in \HH\;\mid\;\scal{x}{o_k} \in C_k\}\label{eq:Csep}
\end{equation}
where $(C_k)_{k\in \KK}$ are closed intervals in $\RR$ such that
$(\forall k \in \KK)$ $C_k \cap \dom\varphi_k \neq \emp$.\\
Suppose that either $\KK$ is finite, or there exists a subset $\LL$ of $\KK$ 
such that:
\begin{enumerate}
\item
\label{p:quezoniva}
$\KK\smallsetminus \LL$ is finite;
\item
\label{p:quezonivb}
$(\forall k\in\LL)$ $\varphi_k\geq \varphi_k(0)=0$ and $0 \in C_k$.
\end{enumerate}
Then, 
\begin{equation}
(\forall x\in\HH)\quad\prox_{\ic+f}{x}=
\sum_{k\in\KK}\pi_k o_k
\label{eq:ressep}
\end{equation}
where 
\begin{equation}
\pi_k = \begin{cases}
\inf C_k & \mbox{if $\prox_{\varphi_k}\scal{x}{o_k} < \inf C_k$}\\
\sup C_k & \mbox{if $\prox_{\varphi_k}\scal{x}{o_k} > \sup C_k$}\\
\prox_{\varphi_k}\scal{x}{o_k} & \mbox{otherwise.}\\
\end{cases}
\label{eq:ressep2}
\end{equation}
\end{proposition}
\begin{proof}
Due to the form of $f$ and $C$, one can write,
\begin{eqnarray*}
(\forall x\in \HH) \qquad \big(f + \iota_{C}\big)(x) = \sum_{k\in \KK}
(\varphi_k+\iota_{C_k})(\langle x,o_k \rangle).
\end{eqnarray*} 
For every $k\in \KK$, $\varphi_k + \iota_{C_k} \in \Gamma_0(\RR)$ 
since $\varphi_k \in \Gamma_0(\RR)$ and
$C_k$ is assumed to be a closed convex set having a nonempty intersection with
$\dom \varphi_k$.
If $\KK$ is not finite, in view of Assumption~\ref{p:quezonivb},
we have 
$(\forall k \in \LL)$ $\varphi_k + \iota_{C_k} \ge (\varphi_k+i_{C_k})(0) = 0$.
From  {\rm\cite[Remark~3.2(ii) and Proposition~2.10]{Chaux_C_2007_ip_variational_ffbip}}, it can be deduced that
\begin{equation}
(\forall x\in \HH) \qquad
\prox_{f + \iota_{C}} x = \sum_{k\in \KK} 
\big(\prox_{\varphi_k+\iota_{C_k}}\langle x,o_k \rangle\big) o_k.
\label{eq:a11}
\end{equation}
On the other hand,
since for every $k\in \KK$, $C_k$ is a closed interval in $\RR$ such that $C_k \cap \dom \varphi_k \neq \emp$, it follows from Proposition~\ref{prop:jstsp}\ref{prop:jstspii}, that
\begin{align}
\prox_{\varphi_k + \iota_{C_k}} \langle x,o_k \rangle = &
(P_{C_k} \circ \prox_{\varphi_k}) (\langle x,o_k \rangle)\nonumber\\
 = &
\begin{cases}		
\inf C_k, &\mbox{if $\prox_{\varphi_{k}} \langle x,o_k \rangle <\inf C_k$}\\
\prox_{\varphi_k} \langle x,o_k \rangle, &\mbox{if $\prox_{\varphi_{k}} \langle x,o_k \rangle \in C_k$}\\
\sup C_k, &\mbox{if $\prox_{\varphi_{k}} \langle x,o_k \rangle  >\sup C_k$}.	
\end{cases}
\label{eq:a12}
\end{align}
Combining \eqref{eq:a11} and \eqref{eq:a12} yields \eqref{eq:ressep} and \eqref{eq:ressep2}.
\end{proof}

A function $f$ (resp. convex $C$) satisfying \eqref{eq:fsep} (resp. \eqref{eq:Csep}) will be said \emph{separable}. Note that \eqref{eq:ressep} and \eqref{eq:ressep2} imply that \eqref{eq:ppcp} holds.
However, this relation has been proved under the restrictive assumption that both $f$ \emph{and} $C$ are separable. In general, when either $f$ \emph{or} $C$ is not separable, \eqref{eq:ppcp}
is no longer valid. Let us give two simple counterexamples to illustrate this fact.

\begin{example}\label{ex:ce1}
Let $\HH = \RR^2$ and $f$ be the function defined by $(\forall x \in \RR^2)$
$f(x) = \frac{1}{2}x^\top\Lambda x$ with 
$\Lambda = \begin{pmatrix} 1&\Lambda_{1,2} \\ \Lambda_{1,2} & \Lambda_{2,2} \end{pmatrix}$
where $\Lambda_{2,2} \ge 0$ and $|\Lambda_{1,2}| \le \Lambda_{2,2}^{1/2}$.
Let $C=[-1,1]^2$. This convex set is separable w.r.t. the canonical basis of $\RR^2$.\\
Now, set $x = 2(\Lambda_{1,2},1+\Lambda_{2,2})^\top$. After some calculations
(see Appendix \ref{ap:exce1}), one obtains:
\begin{itemize}
\item $P_C(\prox_{f}x)=(0,1)^\top$
\item $\prox_{\ic+f}x = (\pi,1)^\top$ where
\begin{equation}
\pi =
\begin{cases}		
\frac{\Lambda_{1,2}}{2} &\mbox{if $\Lambda_{1,2}\in [-2,2]$}\\
1 &\mbox{if $\Lambda_{1,2} > 2$}\\
-1 & \mbox{if $\Lambda_{1,2} < -2$.}	
\end{cases}
\label{eq:defpi}
\end{equation}
\end{itemize}
We conclude that \eqref{eq:ppcp} is not satisfied as soon as $\Lambda_{1,2}\neq 0$, that is 
$f$ is not separable.
\end{example}

\begin{example}\label{ex:ce2}
Let $\HH = \RR^2$.
Consider the separable function defined by
$(\forall x = (x^{(1)},x^{(2)})^\top \in \RR^2)$ $f(x) = (1+\Lambda_{1,2}) (x^{(1)})^2+ (1-\Lambda_{1,2}) 
(x^{(2)})^2$ where $0 < |\Lambda_{1,2}| \le 1$. Let
the nonseparable convex set $C$ be defined by
\begin{equation*}
C = \{x = (x^{(1)},x^{(2)})^\top\in \RR^2\;\mid\; \max(|x^{(1)}-x^{(2)}|,|x^{(1)}+x^{(2)}|) \le \sqrt{2}\}.
\end{equation*}  
In this case, it is shown in Appendix \ref{ap:exce2} that \eqref{eq:ppcp} does not hold. 
\end{example}

In summary, for an arbitrary function in $\Gamma_0(\HH)$ and an arbitrary closed convex set, we cannot trust \eqref{eq:ppcp} to determine 
the proximity operator of the sum of this function
and the indicator function of the convex set.
In the next section, we will propose efficient approaches to compute the desired proximity operator in a general setting.

Other more classical properties of the proximity operator which will be used in the paper are provided in the sequel.
\begin{proposition}\  \label{prop:proxquad}
\begin{enumerate}
 \item \label{p:proxlin} If $f = h + \kappa \scal{\cdot}{x}$
where $h \in \Gamma_0(\HH)$, $x \in \HH$ and $\kappa \in \RR$,
then $\prox_f = \prox_h(\cdot -\kappa x)$.
\item \label{p:proxquad} If $f = h + \vartheta \|\cdot\|^2/2$ where 
$h\in \Gamma_0(\HH)$ and $\vartheta \in \RPP$,
then
\begin{enumerate}
\item \label{p:proxquada} $\prox_f = \prox_{(1+\vartheta)^{-1}h}\big(\cdot/(1+\vartheta)\big)$
\item \label{p:proxquadb} $(\forall (y,z)\in \HH^2)$ $\scal{\prox_f y - \prox_f z}{y-z} 
\ge (1+\vartheta) \|\prox_f y - \prox_f z\|^2$
\item \label{p:proxquadc} $\prox_f$ is strictly contractive\footnote{An operator is strictly contractive
with constant $\beta$ if it is $\beta$-Lipschitz continuous
and $\beta \in ]0,1[$.} with constant $(1+\vartheta)^{-1}$.
\end{enumerate}
\end{enumerate}
\end{proposition}
\begin{proof}
Properties \ref{p:proxlin} and \ref{p:proxquada} result from straightforward calculations \cite[Lemma 2.6]{Combettes_PL_2005_mms_Signal_rbpfbs}. \ref{p:proxquadb} follows from the fact that $\prox_{(1+\vartheta)^{-1}h}$ is firmly nonexpansive \cite[Lemma 2.4]{Combettes_PL_2005_mms_Signal_rbpfbs}, i.e.
\begin{equation*}
(\forall (y,z)\in \HH^2)\qquad \scal{\prox_{\frac{h}{1+\vartheta}} y - \prox_{\frac{h}{1+\vartheta}} z}{y-z} \ge \|\prox_{\frac{h}{1+\vartheta}} y - \prox_{\frac{h}{1+\vartheta}} z\|^2.
\end{equation*}
Thus, by using \ref{p:proxquada}, we have
\begin{align*}
(\forall (y,z)\in \HH^2)\qquad 
&\scal{\prox_f y - \prox_f z}{y-z}\nonumber\\
= &(1+\vartheta) \left\langle\prox_{\frac{h}{1+\vartheta}}\Big(\frac{y}{1+\vartheta}\Big) - 
\prox_{\frac{h}{1+\vartheta}}\Big(\frac{z}{1+\vartheta}\Big), \frac{y}{1+\vartheta}-\frac{z}{1+\vartheta}\right\rangle\nonumber\\
\ge &(1+\vartheta) \left\|\prox_{\frac{h}{1+\vartheta}}\Big(\frac{y}{1+\vartheta}\Big) - \prox_{\frac{h}{1+\vartheta}} \Big(\frac{z}{1+\vartheta}\Big)\right\|^2\nonumber\\
=  &(1+\vartheta)\|\prox_f y - \prox_f z\|^2.
\end{align*}
Property \ref{p:proxquadc} can then be deduced, by invoking the Cauchy-Schwarz inequality:
\begin{align*}
(\forall (y,z)\in \HH^2)\qquad 
(1+\vartheta)\|\prox_f y - \prox_f z\|^2 &\le 
\scal{\prox_f y - \prox_f z}{y-z}\nonumber\\
 &\le \|\prox_f y - \prox_f z\|\|y - z\|.
\end{align*}
\end{proof}

Recall that a function $f\in \Gamma_0(\HH)$ satisfying the assumptions in \ref{p:proxquad} is said to be
strongly convex with modulus $\vartheta$.

\begin{proposition} {\rm \cite[Proposition 11]{Combettes_PL_2007_istsp_Douglas_rsatncvsr}}
\label{p:linprox}
Let $\GG$ be a real Hilbert space, let $f\in\Gamma_0(\GG)$,  and
let $L\colon\HH\to\GG$ be a bounded linear operator.
Suppose that the composition of $L$ and $L^*$ satisfies
$L\circ L^*=\nu\,\Id$, for some $\nu\in\RPP$.
Then $f\circ L\in\Gamma_0(\HH)$ and
\begin{equation}
\label{e:pfL}
\prox_{f\circ L}=\Id+\nu^{-1}L^*\circ(\prox_{\nu f}-\Id)\circ L.
\end{equation}
\end{proposition}

\section{Iterative solutions to the minimization of a sum of two convex functions} \label{se:compprox}
\subsection{Forward-backward approach}
\label{subsec:fb}
Consider the following optimization problem, which is a specialization of 
Problem \ref{pb:minimisation}:
\begin{problem}\label{pb:minimisation1}
Let $f_1$ and $f_2$ be two functions in $\Gamma_0(\HH)$ such that
$\operatorname{Argmin} f_1+f_2 \neq \emp$ and
$f_2$ is differentiable on $\HH$ with a $\beta$-Lipschitz continuous gradient for some $\beta \in \RPP$. 
\begin{equation*}
\text{Find}\quad\min_{x \in \HH} f_1(x) + f_2(x).
\label{eq:forward-backward}
\end{equation*}
\end{problem}

As mentioned in the introduction, the forward-backward algorithm is an effective method
to solve the above problem.

\subsubsection{Algorithm \cite[Eq.(3.6)]{Combettes_PL_2005_mms_Signal_rbpfbs}}
\label{se:fwalgo}
Let $x_0\in \HH$ be an initial value. The algorithm constructs a sequence
$(x_n)_{n\in \NN}$ by setting, for every $n\in\NN$,
\begin{eqnarray}
x_{n+1} = x_n + \lambda_n \big(\prox_{\gamma_n f_1}(x_n - \gamma_n \nabla f_2(x_n) +b_n) + a_n -x_n\big)
\label{eq:fb_algo}
\end{eqnarray} 
where  $\gamma_n >0$ is the algorithm step-size, $\lambda_n > 0$ is a relaxation parameter and $a_n \in \HH$ (resp. $b_n\in \HH$) represents an error allowed in the computation of the proximity operator (resp. the gradient).
The weak convergence of $(x_n)_{n\in \NN}$ to a solution to Problem \ref{pb:minimisation1} is then guaranteed provided that:
\begin{assumption}\hfill \label{a:gl}
\begin{enumerate}
\item \label{a:gli} $0 < \underline{\gamma} \le \overline{\gamma}
< 2\beta^{-1}$ where $\underline{\gamma} =\inf_{n\in \NN} \gamma_n$
and $\overline{\gamma} = \sup_{n\in \NN} \gamma_n$.
\item  \label{a:glii} $(\forall n \in \NN)$ $0<\underline{\lambda}\le \lambda_n \le 1$.
\item \label{a:gliii} $\sum_{n\in \NN} \|a_n\| < +\infty$ and
$\sum_{n\in \NN} \|b_n\| < +\infty$.
\end{enumerate}
\end{assumption}
More details concerning this algorithm can be found in \cite{Combettes_PL_2005_mms_Signal_rbpfbs,Chaux_C_2007_ip_variational_ffbip}
and conditions for the strong convergence of the algorithm are also given in \cite{Combettes_PL_2007_siamopt_proximal_tafmoob}.
An additional result which will be useful in this paper is the following:
\begin{lemma} \label{le:convlinfb}
Suppose that Assumptions \ref{a:gl}\ref{a:gli} and \ref{a:glii}
as well as the assumptions of Problem \ref{pb:minimisation1} hold. If 
$f_1$ is a strongly convex function with modulus $\vartheta$,
then the forward-backward algorithm
in \eqref{eq:fb_algo} with $a_n \equiv b_n \equiv 0$
converges linearly to the unique solution 
$\widetilde{x}$ to Problem \ref{pb:minimisation1}. 
More precisely, we have
\begin{equation}
(\forall n \in \NN)\qquad \|x_n -\widetilde{x}\| \le 
\Big(1-\frac{\underline{\lambda}\underline{\gamma}\vartheta}{1+\underline{\gamma}\vartheta}\Big)^n \|x_0 -\widetilde{x}\|.
\label{eq:linconv}
\end{equation}
\end{lemma}
\begin{proof}
Since $\operatorname{Argmin}f_1+f_2\neq \emp$ and $f_1$ is strongly (thus strictly)
convex, there exists a unique minimizer $\widetilde{x}$ of $f_1+f_2$.
Then, $\widetilde{x}$
is a fixed point of the forward-backward algorithm in \eqref{eq:fb_algo}
when $a_n \equiv b_n \equiv 0$. Thus, we have, for all $n\in \NN$,
\begin{equation*}
x_{n+1} -\widetilde{x} = (1-\lambda_n) (x_n -\widetilde{x})
+ \lambda_n \big(\prox_{\gamma_n f_1}(x_n -\gamma_n \nabla f_2(x_n))
- \prox_{\gamma_n f_1}(\widetilde{x} -\gamma_n \nabla f_2(\widetilde{x}))\big)
\end{equation*}
which yields
\begin{multline*}
\|x_{n+1} -\widetilde{x}\| \le  (1-\lambda_n) \|x_n -\widetilde{x}\|
\\+ \lambda_n \|\prox_{\gamma_n f_1}(x_n -\gamma_n \nabla f_2(x_n))
- \prox_{\gamma_n f_1}(\widetilde{x} -\gamma_n \nabla f_2(\widetilde{x}))\|.
\end{multline*}
Since $f_1$ has been assumed strongly convex with modulus $\vartheta$,
$\gamma_n f_1$ is strongly convex with modulus $\gamma_n \vartheta$
and, according to Assumption \ref{a:gl}\ref{a:gli}, it is also strongly convex
with modulus $\underline{\gamma}\vartheta$. We deduce from Proposition \ref{prop:proxquad}\ref{p:proxquadc} that  $\prox_{\gamma_n f_1}$ is strictly contractive with
constant $(1+\underline{\gamma}\vartheta)^{-1}$. Hence, we have
\begin{equation*}
\|x_{n+1} -\widetilde{x}\| \le  (1-\lambda_n) \|x_n -\widetilde{x}\|
+ \frac{\lambda_n}{1+\underline{\gamma}\vartheta} \|x_n -\gamma_n \nabla f_2(x_n)
- \widetilde{x} +\gamma_n \nabla f_2(\widetilde{x})\|.
\end{equation*}

Recall that an operator $R\;:\;\HH \to \HH$ is nonexpansive if $(\forall (y,z)\in \HH^2)$ $\|R(x)-R(y)\| \le \|x-y\|$.
An operator $T\;:\;\HH \to \HH$ is $\alpha$-averaged with $\alpha \in]0,1[$ if $T =  (1-\alpha) \Id + \alpha R$ where $R$ is a nonexpansive operator.

Since $f_2$ is a differentiable convex function having a $\beta$-Lipschitz continuous gradient with $\beta > 0$, we deduce from the Baillon-Haddad
theorem \cite{Baillon_JP_1977_jm_Quelques_pdoabencm}, that $\nabla f_2 /\beta$ is 1/2-average.
As $\gamma_n \in ]0,2/\beta[$ ,  by using  \cite[Lemma 2.3]{Combettes_PL_2004_o_Solving_mivconao}, $\Id - \gamma_n \nabla f_2$ is $\frac{\gamma_n \beta}{2}$-averaged and it is therefore nonexpansive (see \cite[Lemma 2.1(ii)]{Combettes_PL_2004_o_Solving_mivconao}).

This entails that
\begin{equation*}
\|x_n -\gamma_n \nabla f_2(x_n)
- \widetilde{x} +\gamma_n \nabla f_2(\widetilde{x})\|
\le \|x_n -\widetilde{x}\|
\end{equation*}
and, consequently,
\begin{equation*}
\|x_{n+1} -\widetilde{x}\| \le  \Big(1-\frac{\lambda_n\underline{\gamma}\vartheta}{1+\underline{\gamma}\vartheta}\Big) \|x_n -\widetilde{x}\| \le \Big(1-\frac{\underline{\lambda}\underline{\gamma}\vartheta}{1+\underline{\gamma}\vartheta} \Big)\|x_n -\widetilde{x}\|
\end{equation*}
which results in \eqref{eq:linconv}.
\end{proof}

The linear convergence of the forward-backward algorithm was also proved in \cite{Bredies_K_2007_submitted_Iterative_stcl,Chen_GHG_1997_jopt_Convergence_rifbs} under different assumptions.

\subsubsection{Computation of $\prox_{\ic+\kappa g}$}
\label{sss:picgfb}
Let $\gammat>0$ and $g$ be a differentiable function with $\beta$-Lipschitz continuous gradient where $\beta \in \RPP$. Let $C$ be a closed convex set such that
$C\neq \emp$.
 Then, for every $x \in \HH$, the determination of $\prox_{\ic+\kappa g}x$ can be viewed as a minimization problem of the form of Problem \ref{pb:minimisation1}. Indeed, by using the definition of the proximity operator, we have:
\begin{equation*}
(\forall x\in \HH) \qquad \prox_{\kappa g +\ic}x = \arg \min_{y\in\HH} \frac{1}{2}\left\|y-x\right\|^{2} + \kappa g(y) + \ic(y).
\end{equation*}
Now, we can set $f_1 =  \frac{1}{2}\left\|. -x\right\|^{2}+\ic$ and $f_2 =  \kappa g$.
The proximity operator of $\gamma_n f_1$ with $\gamma_n \in \RPP$, is the proximity operator of $\frac{\gamma_n}{2} \|\cdot\|^2 - \gamma_n\scal{\cdot}{x}+\ic$, which
is straightforwardly deduced 
from Proposition~\ref{prop:proxquad}\ref{p:proxlin} and \ref{p:proxquada}:
\begin{align}
(\forall y \in \HH)
\qquad \prox_{\gamma_n f_1}y 
= P_C\Big(\frac{y+\gamma_n x}{1+\gamma_n}\Big).
\label{eq:defproxf1fw}
\end{align}
whereas $f_2$ has a 
$\kappa \beta$-Lipschitz continuous gradient.
In this case, by setting $a_n \equiv b_n \equiv 0$
in Algorithm \eqref{eq:fb_algo}, we get
\begin{equation}
(\forall n\in \NN) \qquad x_{n+1} = x_n +
 \lambda_n 
\left(P_C\Big(\frac{x_n - \gamma_n ( \kappa \nabla g(x_n)-x)}{1+\gamma_n}\Big)-x_n\right)
 \label{eq:proxdifferentiable}
\end{equation}
with 
\begin{equation}
0 < \underline{\gamma} \le \gamma_n \le \overline{\gamma} < 2\kappa^{-1}\beta^{-1}.
\label{eq:condgammaic}
\end{equation}
The obtained algorithm possesses the following properties:
\begin{proposition} \label{p:fwbc}
Suppose that Condition \eqref{eq:condgammaic} and
Assumption \ref{a:gl}\ref{a:glii} hold. Consider the algorithm in
\eqref{eq:proxdifferentiable} where $x \in \mathcal{H}$.
 Then, 
\begin{enumerate}
\item\label{c1:pfb} 
we have:
\begin{equation}
(\forall n \in \NN) \qquad 
\|x_n - \prox_{\ic+\kappa g}x\| \le \rho^n \|x_0 - \prox_{\ic+\kappa g}x\|
\label{eq:lindecfb}
\end{equation}
where
\begin{equation}
\rho = 1- 
\frac{\underline{\lambda}\underline{\gamma}}{1+\underline{\gamma}}\,;
\label{eq:defrho}
\end{equation}
\item \label{c2:pfb} by setting $x_0=\prox_{\kappa g}x$, we get:
\begin{equation}
\prox_{\kappa g} x \in  C \quad \Rightarrow \quad (\forall n \in \NN)\;\;
x_n = \prox_{\ic+\kappa g}x.
\label{eq:consfwpjstsp}
\end{equation}
\end{enumerate}
\end{proposition}
\begin{proof}
\noindent\ref{c1:pfb} : 
As $f_1$ is strongly convex with modulus 1,
\eqref{eq:lindecfb} is obtained by invoking Lemma~\ref{le:convlinfb}.\\
\ref{c2:pfb} : If $x_0 = \prox_{\kappa g} x \in C$, then \eqref{eq:lindecfb}
leads to
\begin{equation}
(\forall n \in \NN) \qquad 
\|x_n - \prox_{\ic+\kappa g}x\| \le \Big(1- 
\frac{\underline{\lambda}\underline{\gamma}}{1+\underline{\gamma}}\Big)^n 
\| \prox_{\kappa g} x  - \prox_{\ic+\kappa g}x\| = 0
\end{equation}
where Proposition \ref{prop:jstsp}\ref{prop:jstspi} has been used in the last equality. This shows that \eqref{eq:consfwpjstsp} is satisfied.
\end{proof}

\begin{remark}\ \label{r:fw}
\begin{enumerate}
\item Eq. \eqref{eq:lindecfb} shows that 
$(x_n)_{n\in\NN}$ converges linearly to $\prox_{\ic+\kappa g}x$.
Although this equation provides an upper bound, it suggests
to choose $\lambda_n$ and $\gamma_n$ as large as possible (i.e.
$\lambda_n \equiv 1$ and $\gamma_n$ close to $2\kappa^{-1}\beta^{-1}$)
to optimize the convergence rate. This fact was confirmed by our simulations.
\item\label{r:fw1step} Proposition \ref{p:fwbc}\ref{c2:pfb} may appear as
a desirable property since Proposition \ref{prop:jstsp}\ref{prop:jstspi}
states that, when $\prox_{\kappa g} x \in C$, $\prox_{\ic+\kappa g}x$
takes a trivial form.
In this case, the convergence is indeed guaranteed in just one iteration 
by appropriately initializing the algorithm. Note however that $\prox_{\kappa g}x$ may not always be simple to compute, depending on the form
of $g$.
\item An alternative numerical method for the computation of
 $\prox_{\ic+\kappa g}x$ would consist of setting
$f_1 = \ic$ and $f_2 =  \frac{1}{2}\left\|. -x\right\|^{2}+\kappa g$
in the forward-backward algorithm, so yielding
\begin{equation*}
(\forall n\in \NN) \qquad x_{n+1} = x_n +
\lambda_n \big(P_C(x_n - \gamma_n ( \kappa \nabla g(x_n)+x_n-x)) -x_n\big)
\end{equation*}
with  $0 < \underline{\gamma} \le \overline{\gamma} < 2(\kappa\beta+1)^{-1}$.
It can be noticed that the forward-backward algorithm then reduces to a projected gradient algorithm {\rm \cite[Chap. 3., Sect. 3.3.2]{Bertsekas_DP_1997_book_Parallel_adcnm}\cite{Alber_YI_1998_mp_projected_otpsmfncoiahs}}, when $\lambda_n \equiv 1$. In our experiments, it was however
observed that the convergence of this algorithm is slower
than that in \eqref{eq:proxdifferentiable}, probably due to the fact
that $\prox_{\gamma_n f_1}$ is no longer strictly contractive for the second
choice of $f_1$.
\end{enumerate}
\end{remark}

\subsection{Douglas-Rachford approach}
\label{subsec:dr}
Let us relax the Lipschitz continuity assumption in Problem~\ref{pb:minimisation1} 
and turn our attention to the optimization problem:
\begin{problem}\label{pb:minimisation2} 
Let $g_1$ and $g_2$ be functions in $\Gamma_0(\HH)$  such that
$\operatorname{Argmin} g_1+g_2 \neq \emp$.
Assume that one of the following three conditions is satisfied: 
\begin{enumerate}
\item
\label{p:1ii}
$\dom g_2\cap\inte \dom g_1 \neq\emp$.\footnote{The interior (resp. relative interior) of a set $S$ of $\HH$ is designated by $\inte S$ (resp. $\rint S$).}
\item
\label{p:1i}
$\dom g_1\cap\inte \dom g_2 \neq\emp$.
\item \label{p:1iii}
$\HH$ is finite dimensional and $\rint \dom g_1 \cap\rint \dom g_2 \neq\emp$.
\end{enumerate}
\begin{equation*}
\text{Find}\quad \min_{z \in \HH} g_1(z) + g_2(z).
\label{eq:douglas-rachford}
\end{equation*}
\end{problem}
In the statement of the above problem, the notation differs
from that used in Problem~\ref{pb:minimisation1} to emphasize
the difference in the assumptions which have been adopted and 
facilitate the presentation of the algorithms subsequently presented
in Section \ref{sec:algorithm}.\\
The Douglas-Rachford algorithm,
proposed in \cite{Lions_PL_1979_jna_Splitting_aftsotno,Eckstein_J_2003_mp_douglas_otdrsmatppaftmmo}, provides an appealing numerical solution to Problem \ref{pb:minimisation2}, as described next.
\vspace*{0.5cm}
\subsubsection{Algorithm \cite[Eq.(19)]{Combettes_PL_2007_istsp_Douglas_rsatncvsr}}
Set $z_0 \in \HH$ and compute, for every $m\in \NN$,
\begin{equation}
\begin{cases}
z_{m+\frac{1}{2}} = \prox_{\gammat g_2} z_m + b_m\\
z_{m+1} = z_m + \lambdat_m \big(\prox_{\gammat g_1}(2z_{m+\frac{1}{2}} - z_m) + a_m - z_{m+\frac{1}{2}}\big)\\
\end{cases}
\label{eq:algDR}
\end{equation}
where $\kappa >0$, $\suite[m]{\lambdat_m}$ is a sequence of positive reals, and 
$\suite[m]{a_m}$ (resp. $\suite[m]{b_m}$) is a sequence of
 errors in $\HH$ allowed in the computation  of the proximity operator
of $\gammat g_1$ (resp. $\gammat g_2$).\\
Then, $(z_m)_{m\in \NN}$ converges weakly to $z \in \HH$ \cite[Corollary 5.2]{Combettes_PL_2004_o_Solving_mivconao} such that $\prox_{\gammat g_2}z$ is a solution to Problem~\ref{pb:minimisation2}, provided that: 

\begin{assumption}\hfill \label{a:gl2}
\begin{enumerate}
 \item \label{a:gl21} $(\forall m \in \NN)$ $\lambdat_m\in ]0,2[$ and
$\sum_{m\in\NN} \lambdat_m(2-\lambdat_m) = +\infty$.
 \item \label{a:gl22} $\sum_{m\in\NN} \lambdat_m(\norm{a_m}+\norm{b_m})<+\infty$.
\end{enumerate}
\end{assumption}
\vspace*{0.3cm}
An alternate convergence result is the following:
\begin{proposition}\label{prop:constrongDR}
Suppose that 
the assumptions
of Problem \ref{pb:minimisation2} hold. If $g_2$ is a strongly convex function, then the Douglas-Rachford algorithm in \eqref{eq:algDR} with $\inf_{m \in \NN}\tau_m > 0$,
$\sup_{m\in\NN} \tau_m \le 2$
and $a_m \equiv b_m \equiv 0$ is such that $(z_{m+1/2})_{m\in \NN}$ converges strongly  to the unique solution
to Problem  \ref{pb:minimisation2}.
\end{proposition}
\begin{proof}
Let the $\rprox$ operator be defined, for every $f\in \Gamma_0(\HH)$,  by 
\begin{equation}
\rprox_f = 2\prox_f - \Id.
\end{equation}
Let us rewrite the Douglas-Rachford iteration in \eqref{eq:algDR} with $a_m \equiv b_m \equiv 0$ 
as $z_{m+1} = S_m z_m$, where
\begin{equation}
S_m = \lambdat_m \prox_{\gammat g_1}(\rprox_{\gammat g_2})+ \Id - \lambdat_m \prox_{\gammat g_2}.
\label{eq:S}
\end{equation}
For all $(\yy,\yyy)\in\HH^2$, we have
\begin{equation}
\begin{split}
\normCar{&S_m\yy-S_m\yyy} = \lambdat_m^{2}\normCar{\prox_{\gammat g_1}(\rprox_{\gammat g_2}\yy) -\prox_{\gammat g_1}(\rprox_{\gammat g_2}\yyy)}\\ 
&+2 \lambdat_m \prodScal{\prox_{\gammat g_1}(\rprox_{\gammat g_2}\yy)-\prox_{\gammat g_1}(\rprox_{\gammat g_2}\yyy)}{\yy-\lambdat_m\prox_{\gammat g_2}\yy  -\yyy + \lambdat_m \prox_{\gammat g_2}\yyy}\\
& + \normCar{\yy-\lambdat_m\prox_{\gammat g_2}\yy  -\yyy + \lambdat_m \prox_{\gammat g_2}\yyy}.
\end{split}
\label{strgCV1}
\end{equation}
Since $\gammat g_1 \in \Gamma_0(\HH)$,  $\prox_{\gammat g_1}$ is firmly nonexpansive
\cite[Lemma 2.4]{Combettes_PL_2005_mms_Signal_rbpfbs}
and, the expression in \eqref{strgCV1} can be upper bounded as follows
\begin{equation*}
\begin{split}
&\normCar{S_m\yy-S_m\yyy}\\
& \leq  \lambdat_m^{2}\prodScal{\prox_{\gammat g_1}(\rprox_{\gammat g_2}\yy) -\prox_{\gammat g_1}(\rprox_{\gammat g_2}\yyy)}{\rprox_{\gammat g_2}\yy -
\rprox_{\gammat g_2}\yyy}\\
 & { +2 \lambdat_m \prodScal{\prox_{\gammat g_1}(\rprox_{\gammat g_2}\yy)-\prox_{\gammat g_1}(\rprox_{\gammat g_2}\yyy)}{\yy-\lambdat_m\prox_{\gammat g_2}\yy  -\yyy + \lambdat_m \prox_{\gammat g_2}\yyy}}\\
& + \normCar{\yy-\lambdat_m\prox_{\gammat g_2}\yy  -\yyy + \lambdat_m \prox_{\gammat g_2}\yyy}
\end{split}
\label{strgCV2}
\end{equation*}
which yields after simplifications:
\begin{equation*}
\begin{split}
\normCar{S_m\yy-S_m\yyy} \leq  \lambdat_m(2 - \lambdat_m)&\prodScal{\prox_{\gammat g_1}(\rprox_{\gammat g_2}\yy) -\prox_{\gammat g_1}(\rprox_{\gammat g_2}\yyy)}{\yy-\yyy}\\
& + \normCar{\yy-\lambdat_m\prox_{\gammat g_2}\yy  -\yyy + \lambdat_m \prox_{\gammat g_2}\yyy}.
\end{split}\label{strgCV3}
\end{equation*}
Using the definition of the operator $S_m$ in \eqref{eq:S},
we thus obtain, after some simple calculations,
\begin{multline}
\normCar{S_m\yy-S_m\yyy} \leq (2-\lambdat_m)\prodScal{S_m\yy-S_m\yyy}{\yy-\yyy} + (\lambdat_m-1)\normCar{\yy-\yyy} \\
- \lambdat_m ^{2}\big(\prodScal{ \prox_{\gammat g_2}\yy - \prox_{\gammat g_2}\yyy }{\yy-\yyy}- \normCar{\prox_{\gammat g_2}\yy - \prox_{\gammat g_2}\yyy}\big).
\label{eq:lasteqy1y2str}
\end{multline}
Let $\theta$ be the modulus of the strongly convex function $g_2$. Then 
$\gammat g_2$ is strongly convex with modulus $\gammat \theta$ and 
Proposition \ref{prop:proxquad}\ref{p:proxquadb} states that the following inequality holds:
\begin{equation*}
 \prodScal{ \prox_{\gammat g_2}\yy - \prox_{\gammat g_2}\yyy }{\yy-\yyy} \geq (\gammat \theta +1)\normCar{\prox_{\gammat g_2}\yy - \prox_{\gammat g_2}\yyy},
\end{equation*}
which combined with \eqref{eq:lasteqy1y2str} leads to
\begin{equation}
\begin{split}
\normCar{S_m\yy-S_m\yyy} + &\gammat \theta\lambdat_m ^{2}\normCar{\prox_{\gammat g_2}\yy - \prox_{\gammat g_2}\yyy} \\
&\leq (2-\lambdat_m)\prodScal{S_m\yy-S_m\yyy}{\yy-\yyy} + (\lambdat_m-1)\normCar{\yy-\yyy}.
\end{split}
\label{strgCV8}
\end{equation}

Now, let $\zz$ be the unique minimizer of $g_1+g_2$.
Hence, $\zz = \prox_{\gammat g_2} z$ where $z$ is a fixed point of $S_m$.
Consequently, by setting  $\yy = z_m$ and $\yyy = z$ in \eqref{strgCV8}, we deduce that
\begin{multline}
\normCar{z_{m+1}- z} + \gammat \theta\lambdat_m ^{2}\normCar{z_{m+\frac{1}{2}} - \zz} \\
 \leq (2-\lambdat_m)\prodScal{z_{m+1}- z}{z_m-z} + (\lambdat_m-1)\normCar{z_m-z}.
\label{eq:strgCV8bis}
\end{multline}
Using the fact that 
\[
2\prodScal{z_{m+1}-z}{z_m-z} =  \normCar{z_{m+1}-z} + \normCar{z_m-z} -\normCar{z_{m+1}-z_m}
\] 
\eqref{eq:strgCV8bis} can be rewritten as
\begin{equation}
\begin{split}
\lambdat_m \normCar{z_{m+1}- z} &+(2-\lambdat_m)\normCar{z_{m+1} - z_m} + 2\gammat \theta\lambdat_m ^{2}\normCar{z_{m+\frac{1}{2}} - \zz} \leq \lambdat_m \normCar{z_m-z}.
\end{split}
\label{strgCV12}
\end{equation}
Considering Assumption \ref{a:gl2}, 
$(2-\lambdat_m)$ $\normCar{z_{m+1} - z_m}$ is nonnegative and the left-hand side term of inequality \eqref{strgCV12} can be lower bounded, so yielding
\begin{equation*}
\begin{split}
\lambdat_m \normCar{z_{m+1}- z} + 2\gammat \theta\lambdat_m ^{2}\normCar{z_{m+\frac{1}{2}} - \zz} \leq \lambdat_m \normCar{z_m-z}.
\end{split}
\label{strgCV13}
\end{equation*}
Finally, by using the assumption that 
$\underline{\lambdat}=\inf_{m\in \NN} \lambdat_m > 0$, we obtain
\begin{equation}
\begin{split}
\normCar{z_{m+1}- z} + 2\gammat \theta\underline{\lambdat} \normCar{z_{m+\frac{1}{2}} - \zz} \leq \normCar{z_m-z}.
\end{split}
\label{strgCV14}
\end{equation}
This entails that $\normCar{z_{m+1}-z}\leq \normCar{z_m - z}$
and, the sequence $\suite{\|z_m-z\|}$ being decreasing, there exists $c \in \RPP$ such that $\lim_{m\rightarrow\pinf}\|z_m-z\|=c$.
In turn, from \eqref{strgCV14}, we conclude that $\lim_{m\rightarrow\pinf}z_{m+\frac{1}{2}} = \zz$, which shows the strong convergence of $(z_{m+1/2})_{m\in\NN}$
to the unique minimizer of $g_1+g_2$.
\end{proof}

It can be noticed that, although the convergence of the Douglas-Rachford
algorithm generally requires that $\tau_m < 2$, the strong convergence is obtained under the above assumptions, when
$\tau_m = 2$. The limit case of the Douglas-Rachford corresponding
to $\tau_m \equiv 2$
is known as the Peaceman-Rachford algorithm \cite{Peaceman_DW_1955_siam_numerical_tnsopaede,Combettes_PL_2004_o_Solving_mivconao}.

\subsubsection{Computation of $\prox_{\ic+ \gamma\lowercase{f}}$} \label{sec:icfDR}
Let $C$ be a nonempty closed convex set of $\HH$.
The Douglas-Rachford algorithm can be used to compute $\prox_{\ic+ \gamma f}$ where $f \in \Gamma_0(\HH)$ and $\gamma$ is a positive constant, 
using again the definition of the proximity operator:
\begin{equation}
(\forall x \in \HH)\quad \prox_{\ic +  \gamma f}x = \arg \min_{y\in \HH}\frac{1}{2}\left\|y-x\right\|^{2} + \ic(y)+  \gamma f(y).
\label{eq:prox_definition}
\end{equation}
The above minimization problem 
appears as a specialization of Problem \ref{pb:minimisation2} by setting
$g_1 = \gamma f$ and $g_2 = \frac{1}{2}\left\|\cdot-x\right\|^{2} + \ic$,
provided that one of the following three conditions holds:
\begin{assumption}\hfill \label{a:ddd}
\begin{enumerate}
\item \label{a:ddd1}$  C \cap \inte \dom f \neq \emp$.
\item \label{a:ddd2}$ \dom f  \cap \inte C \neq \emp$.
\item \label{a:ddd3}$\HH$ is finite dimensional and $\rint C  \cap  \rint\dom f \neq \emp$.
\end{enumerate}
\end{assumption}
Subsequently, we propose to use the 
Douglas-Rachford algorithm in \eqref{eq:algDR} with $a_m \equiv b_m \equiv 0$,
to compute the desired proximity operator.
Note that both $\prox_{\gammat g_1}$ and $\prox_{\gammat g_2}$ with $\gammat > 0$, 
have to be calculated to apply this algorithm. In our case, we have
\begin{eqnarray*}
\prox_{\gammat g_1} = \prox_{\gammat \gamma f}
\label{eq:pf1}
\end{eqnarray*}
and, similarly to \eqref{eq:defproxf1fw},
\begin{eqnarray*}
(\forall z \in \HH)\qquad
\prox_{\gammat g_2}z = P_C \Big(\frac{z+\gammat x}{1+\gammat}\Big).
\label{eq:pf2}
\end{eqnarray*}
The resulting Douglas-Rachford iterations read: for every $m\in \NN$,
\begin{equation}
\begin{cases}
\displaystyle z_{m+\frac{1}{2}} =  P_C \Big(\frac{z_m+\gammat x}{1+\gammat}\Big)\\
z_{m+1} = z_m + \lambdat_m 
\big(\prox_{\gammat \gamma f}(2z_{m+\frac{1}{2}} - z_m) - z_{m+\frac{1}{2}}\big).\label{eq:algDR2}
\end{cases}
\end{equation}
This algorithm enjoys the following properties:
\begin{proposition} \label{prop:DRn}
Suppose that one of Assumptions  \ref{a:ddd}\ref{a:ddd1}, \ref{a:ddd}\ref{a:ddd2}
or \ref{a:ddd}\ref{a:ddd3} holds.
Consider the algorithm in \eqref{eq:algDR2}
where $x\in \HH$,
$\inf_{m \in \NN}\tau_m>0$ and $\sup_{m \in \NN}\tau_m\le 2$.
Then,
\begin{enumerate}
\item \label{prop:DRni}
 $(z_{m+\frac{1}{2}})_{m\in \NN}$ converges strongly to $\prox_{\ic + \gamma f}x$;
\item \label{prop:DRni2}by setting $\kappa = 1$ and $z_0 = 2\prox_{\gamma f} x-x$, we get:
\begin{equation}
\prox_{\gamma f} x \in C \quad \Rightarrow\quad (\forall m \in \NN)\;\;z_{m+\frac{1}{2}} = \prox_{\ic +\gamma  f}x.
\label{eq:1stepDR}
\end{equation}
\end{enumerate}
\end{proposition}
\begin{proof}
\noindent\ref{prop:DRni}: As $g_2$ is strongly convex with modulus 1, \ref{prop:DRni} holds by invoking Proposition \ref{prop:constrongDR}.\\
\ref{prop:DRni2}: Set $\gammat=1$, $z_0 = 2\prox_{\gamma f}x-x$ with $\prox_{\gamma f}x\in C$. By considering the first iteration of the Douglas-Rachford algorithm ($m=0$), we have $z_{\frac{1}{2}} = \prox_{\gamma f}x$ and $z_1 = z_0$. 
So, by induction, $(\forall m \in \NN)$ 
$ z_{m+\frac{1}{2}} = \prox_{\gamma f}x$, which is also equal to
$\prox_{\ic + \gamma f}x$ according to Proposition~\ref{prop:jstsp}\ref{prop:jstspi}.
\end{proof}
\begin{remark}\ 
\begin{enumerate}
\item 
As already observed in Remark \ref{r:fw}\ref{r:fw1step}, \eqref{eq:1stepDR}
is a desirable property.
It shows that the proposed algorithm converges in one iteration when $\prox_{\gamma f} x \in C$, which appears quite consistent in the light of
Proposition \ref{prop:jstsp}\ref{prop:jstspi}.
\item  Other choices can be envisaged for $g_1$ and $g_2$, namely
\begin{enumerate}
\item $g_1 = \frac{1}{2}\left\|\cdot-x\right\|^{2} + \ic$ and $g_2 = \gamma f$
\item $g_1 =\frac{1}{2}\left\|\cdot-x\right\|^{2}+\gamma f$ and $g_2=\ic$
\item $g_1 =\ic$ and $g_2=\frac{1}{2}\left\|\cdot-x\right\|^{2}+\gamma f$.
\end{enumerate}
Nevertheless, the strong convergence of $(z_{m+1/2})_{m\in\NN}$ in virtue of Proposition~\ref{prop:constrongDR}
is only guaranteed in the third case, whereas Property~\eqref{eq:1stepDR} holds only in the first case (when $\kappa = 1$ and $z_0 = x$). The second case was investigated in \cite{Dupe_FX_2008_ip_proximal_ifdpniusr}, where the good numerical behaviour of the resulting algorithm was demonstrated.
\end{enumerate}
\end{remark}

\subsection{Discussion}
Both Algorithms \eqref{eq:proxdifferentiable} and \eqref{eq:algDR2} allow us to determine the proximity operator of the sum of the indicator function of a closed convex set and a function in $\Gamma_0(\HH)$. The main difference between the two methods is that, in the former one, the convex function needs to be differentiable with a Lipschitz-continuous gradient, whereas the latter requires that the proximity operator of the convex function is easy to compute. In addition, the forward-backward algorithm converges linearly, while we were only able to prove the strong convergence of the Douglas-Rachford algorithm.
As we have shown also, the two algorithms are consistent with Proposition~\ref{prop:jstsp}\ref{prop:jstspi}.

\section{Proposed algorithms to minimize $f+g+\ic$} \label{sec:algorithm}
We have presented two approaches to minimize the sum of two functions in $\Gamma_0(\HH)$. We have also seen that these methods can be employed
to compute the proximity operator of the sum of the indicator function of a closed convex set $C$ and a function in $\Gamma_0(\HH)$.

We now come back to the more general form of Problem \ref{pb:minimisation},
for which we will propose two solutions. Both of them correspond to a combination of the forward-backward algorithm and the Douglas-Rachford one.

\subsection{First method: insertion of a forward-backward step in the Dou\-glas-Rachford algorithm} 
\label{sec:DR(FB)}
We propose to apply the Douglas-Rachford algorithm as described in Section \ref{subsec:dr}, when $g_1 = f$ and $g_2= \ic+ g$. If we refer to \eqref{eq:algDR}, we need to determine $\prox_{\kappa g_1} = \prox_{\kappa f}$ and $\prox_{\kappa g_2} = \prox_{\ic+\kappa g}$, where $\kappa > 0$. The main difficulty lies in the computation of the second proximity operator.
As proposed in Section \ref{sss:picgfb},
we can use a forward-backward algorithm to achieve this goal. 
The resulting algorithm is:

\begin{algorithm}\hfill \label{algo:main1}
\begin{itemize}
\item[\Pisymbol{pzd}{192}] Set $\underline{\gamma}\in ]0,2\gammat^{-1} \beta^{-1}[$,
$\underline{\lambda} \in ]0,1]$ and $\kappa \in \RPP$.
Choose $(\tau_m)_{m\in \NN}$ satisfying Assumption \ref{a:gl2}\ref{a:gl21}.
\item[\Pisymbol{pzd}{193}] Set $m=0$, $z_0 = z_{-1/2} \in C$.
\item[\Pisymbol{pzd}{194}] Set $x_{m,0} = z_{m-1/2}$.
\item[\Pisymbol{pzd}{195}] For $n = 0,\ldots,N_m-1$
\begin{itemize}
\item[a)] Choose $\gamma_{m,n}\in [\underline{\gamma},2\gammat^{-1} \beta^{-1}[$
and $\lambda_{m,n} \in [\underline{\lambda},1]$.
\item[b)] Compute
\[
 x_{m,n+1} = x_{m,n} + \lambda_{m,n} \left(
P_C \Big( \frac{x_{m,n} - \gamma_{m,n} (\gammat \nabla g(x_{m,n})-z_m)}{1+\gamma_{m,n}}\Big) - x_{m,n}\right).
\]
\end{itemize}
\item[\Pisymbol{pzd}{196}] Set $z_{m+\frac{1}{2}} = x_{m,N_m}$.
\item[\Pisymbol{pzd}{197}] Set $z_{m+1} = z_m + \lambdat_m \big(\prox_{\kappa f}(2z_{m +\frac{1}{2}}-z_m) - z_{m+\frac{1}{2}} \big)$.
\item[\Pisymbol{pzd}{198}] Increment $m$ $(m \leftarrow m+1)$
and goto \Pisymbol{pzd}{194}.
\end{itemize}  
\end{algorithm}

Step \Pisymbol{pzd}{192} allows us to set the algorithm parameters and
Step \Pisymbol{pzd}{193} corresponds to the initialization of the algorithm.
At iteration $m\ge 0$,
Step \Pisymbol{pzd}{195} consists of $N_m \ge 1$ iterations
of the forward-backward part of the algorithm, where possibly varying step-sizes $(\gamma_{m,n})_n$ and relaxation parameters $(\lambda_{m,n})_n$ are used.
Finally Steps \Pisymbol{pzd}{196} and \Pisymbol{pzd}{197} correspond to the Douglas-Rachford iteration. Here, the error term $a_m$ 
in the computation of $\prox_{\kappa f}$ is assumed to be equal to zero
but, due to the finite
number of iterations $N_m$ performed in Step \Pisymbol{pzd}{195}, an error
$b_m = z_{m+1/2} - \prox_{\ic+\kappa g}z_m$ may be introduced
in Step \Pisymbol{pzd}{195}.

It can be noticed that the forward-backward algorithm has not been initialized
in Step \Pisymbol{pzd}{194}
as suggested by Proposition \ref{p:fwbc}\ref{c2:pfb}. Indeed, as already mentioned, the computation of $\prox_{\kappa g} z_m$ would be generally costly. Furthermore, the initialization in Step \Pisymbol{pzd}{194} is useful to guarantee the following properties:
\begin{proposition}\label{p:convfwDR}
Suppose that Problem \ref{pb:minimisation} has a solution and
that one of Assumptions \ref{a:ddd}\ref{a:ddd1}, \ref{a:ddd}\ref{a:ddd2} or
\ref{a:ddd}\ref{a:ddd3} holds.
\begin{enumerate}
\item\label{p:iconvfwDR} Let $\xi > 0$ and $\rho$ be given by \eqref{eq:defrho}.
If $\inf g(C) > -\infty$ and, for every $m\in \NN$, the positive integer $N_m$ is chosen such that
\begin{subequations}
\begin{align}
&\rho^{N_m} \sqrt{2\kappa} \big(g(z_0)-\inf g(C)\big)^{1/2} \le \xi
&\mbox{if $m=0$}
\label{eq:condconvfwDR0}\\
&\rho^{N_m-1} \big(1+\xi^{-1}\rho^{1-m} \|z_m -z_{m-1}\|\big) \le 1
&\mbox{if $m > 0$}
\label{eq:condconvfwDRm}
\end{align}
\end{subequations}
then, $(z_m)_{m\in \NN}$ converges weakly to $z \in \HH$ such that $\prox_{\iota_C+\gammat g}z$ is solution to Problem \ref{pb:minimisation}.
\item \label{p:iiconvfwDR} For every $m\in \NN$, $(x_{m,n})_{0\le n \le N_m}$ (and thus, $z_{m+1/2}$)
lies in $C$.
\end{enumerate}
\end{proposition}
\begin{proof}
\noindent\ref{p:iconvfwDR}:
According to Proposition \ref{p:fwbc}\ref{c1:pfb},
for every $m\in \NN$,
\begin{equation*}
(\forall n \in \{0,\ldots,N_m\})\qquad
\|x_{m,n}- \prox_{\ic + \kappa g} z_m \|
\le \rho^n \|x_{m,0}- \prox_{\ic + \kappa g} z_m \|
\end{equation*}
and, consequently
\begin{equation}
\|b_m \| = \|z_{m+1/2} -\prox_{\ic + \kappa g} z_m\|
\le \rho^{N_m}  \|z_{m-1/2} -\prox_{\ic + \kappa g} z_m\|.
\label{eq:upboundbm}
\end{equation}
Let us next show by induction that Conditions \eqref{eq:condconvfwDR0}
and \eqref{eq:condconvfwDRm} allow us to guarantee
that
\begin{equation}
\|b_m \| \le \rho^m \xi.
\label{eq:geobm}
\end{equation}
\begin{itemize}
\item If $m = 0$, we deduce from \eqref{eq:upboundbm} that
\begin{equation}
\|b_0 \| \le \rho^{N_0}  \|z_0 -\prox_{\ic + \kappa g} z_0\|.
\label{eq:upboundb0}
\end{equation}
From the definition of the proximity operator, we have
\begin{align*}
(\forall x \in C)\qquad
&\frac{1}{2} \|z_0-x\|^2 + \kappa\, g(x)\nonumber\\
&\ge \frac{1}{2} \| z_0-\prox_{\ic + \kappa g} z_0\|^2 + 
\kappa \,g\big(\prox_{\ic + \kappa g} z_0\big)\nonumber\\
& \ge \frac{1}{2} \| z_0-\prox_{\ic + \kappa g} z_0 \|^2 + \kappa \,\inf g(C)
\end{align*}
and, since $z_0 \in C$,
\begin{equation*}
\| z_0 - \prox_{\ic + \kappa g} z_0 \|^2
\le 2\kappa \big(g(z_0)-\inf g(C)\big).
\end{equation*}
By combining the latter inequality with \eqref{eq:upboundb0}
and \eqref{eq:condconvfwDR0}, we conclude that $\|b_0\| \le \xi$.
\item Now, let us show that \eqref{eq:geobm} holds for $m>0$, by assuming that
$\|b_{m-1}\| \le \rho^{m-1} \xi$. Using \eqref{eq:upboundbm}, we have
\begin{align*}
\|b_m\| 
& \le \rho^{N_m}  \big(\|z_{m-1/2} -\prox_{\ic + \kappa g} z_{m-1}
+\prox_{\ic + \kappa g} z_{m-1}-\prox_{\ic + \kappa g} z_m\|\big)\nonumber\\
&  \le \rho^{N_m}  \big(\|b_{m-1}\| + \|\prox_{\ic + \kappa g} z_{m-1}-\prox_{\ic + \kappa g} z_m\|\big)
\nonumber\\
& \le \rho^{N_m}  \big(\|b_{m-1}\| + \|z_{m-1}-z_{m}\|\big)
\end{align*}
where the nonexpansivity of $\prox_{\ic+\kappa g}$ has been used in the last
inequality.
From the induction assumption, we deduce that
\begin{equation*}
\|b_m\| \le \rho^{N_m} (\rho^{m-1} \xi + \|z_{m-1}-z_{m}\|\big)
\end{equation*}
which, according to \eqref{eq:condconvfwDRm}, leads to \eqref{eq:geobm}.
\end{itemize}
Then, \eqref{eq:geobm} allows us to claim that Assumption \eqref{a:gl2}\ref{a:gl22}
is satisfied since
\begin{equation*}
\sum_{m\in \NN} \tau_m (\|a_m\|+ \|b_m\|) \le
2 \xi (1-\rho)^{-1}.
\end{equation*}
By further noticing that Assumption \ref{a:ddd}
is equivalent to \ref{a:ddd1} $( \dom(\ic+g) \cap \inte \dom f \neq \emp)$,
\ref{a:ddd2} $\left(\dom f \cap \inte \dom(\ic+g) \neq \emp\right)$
or, \ref{a:ddd3} $\HH$ is finite dimensional and
 $(\rint \dom f \cap \rint\dom(\ic+g) \neq \emp)$,
the conditions for the weak convergence of the Douglas-Rachford algorithm
are therefore fulfilled.\\
\ref{p:iiconvfwDR}: The property can be proved by induction by noticing that
$x_{0,0} = z_{-1/2} \in C$ and that $x_{m,n+1}$ is a convex combination
of $x_{m,n}$ and the projection onto $C$ of an element of $\HH$.\end{proof}

Eqs. \eqref{eq:condconvfwDR0} and \eqref{eq:condconvfwDRm} constitute more a theoretical
guaranty for the convergence of the proposed algorithm than a practical guideline for the choice of $N_m$. In our numerical experiments, these conditions were indeed observed to provide overpessimistic values
of the number of forward-backward iterations to be applied in
Step \Pisymbol{pzd}{195}.
 
As a consequence of Proposition \ref{p:convfwDR}\ref{p:iiconvfwDR},
in Step \Pisymbol{pzd}{195}b), the gradient of $g$ is only evaluated on $C$. 
This means that the assumption of Lipschitz-continuity on the gradient of $g$ is only required on $C$ and therefore, the algorithm can be applied to the following more general setting:
\begin{problem} \label{prob:minimisationgen}
Let $C$ be a nonempty closed convex subset of $\HH$. Let $f$ and $g$ be in $\Gamma_{0}(\HH)$, where $g$ is differentiable on $C$ with a $\beta$-Lipschitz continuous gradient for some $\beta\in\RPP$.\footnote{That is there exists an open set containing $C$ on which $g$ is differentiable with a $\beta$-Lipschitz continuous gradient.}
\begin{equation*}
\text{Find}\qquad \min_{x\in C} f(x)+g(x).
\end{equation*}
\end{problem}
Note that, in the latter problem, the function $g$ does need to be finite.

\subsection{Second method: insertion of a Douglas-Rachford step in the for\-ward-backward algorithm} \label{sec:FB(DR)}
For this method, a different association between the functions involved in Problem \ref{pb:minimisation} is considered by setting $f_1 = \ic+ f$ and $f_2= g$. Since $f_2$ has then a $\beta$-Lipschitz continuous gradient, we can apply the forward-backward algorithm presented in Section \ref{se:fwalgo}. This requires however to compute $\prox_{\gamma_n f_1} = \prox_{\ic+\gamma_n f}$, which can be performed with Douglas-Rachford iterations.

Let us summarize the complete form of the second algorithm we propose to solve Problem \ref{pb:minimisation}. 
\begin{algorithm}\hfill \label{algo:main}
\begin{itemize}
\item[\Pisymbol{pzd}{192}] Choose sequences $(\gamma_n)_{n\in \NN}$ and $(\lambda_n)_{n\in\NN}$ satisfying Assumptions \ref{a:gl}\ref{a:gli} and \ref{a:glii}.
Set $\underline{\tau} \in ]0,2]$.
\item[\Pisymbol{pzd}{193}] Set $n=0$, $x_0 \in C$.
\item[\Pisymbol{pzd}{194}] Set $x_n' = x_n - \gamma_n \nabla g(x_n)$.
\item[\Pisymbol{pzd}{195}] Set $z_{n,0} = 2 \prox_{\gamma_n f} x_n'-x_n'$.
\item[\Pisymbol{pzd}{196}] For $m = 0,\ldots,M_n-1$
\begin{itemize}
\item[a)] Compute $\displaystyle z_{n,m+\frac{1}{2}} = P_C\Big(\frac{z_{n,m} +x_n'}{2}\Big)$.
\item[b)] Choose $\tau_{n,m} \in [\underline{\tau},2]$.
\item[c)] Compute $z_{n,m+1} = z_{n,m} +\tau_{n,m}\big(\prox_{\gamma_n f}(2z_{n,m+\frac{1}{2}} - z_{n,m}) -z_{n,m+\frac{1}{2}}\big)$.
\item[d)] If $z_{n,m+1} = z_{n,m}$, then goto \Pisymbol{pzd}{197}.
\end{itemize}
\item[\Pisymbol{pzd}{197}] Set $x_{n+1} = x_n + \lambda_n \big(z_{n,m+\frac{1}{2}}-x_n\big)$.
\item[\Pisymbol{pzd}{198}] Increment $n$ $(n\leftarrow n+1)$ and goto \Pisymbol{pzd}{194}.
\end{itemize}  
\end{algorithm}

We see that Step \Pisymbol{pzd}{196} consists of at most $M_n \ge 1$ iterations of the Douglas-Rachford algorithm described in Section \ref{sec:icfDR},
which is initialized in accordance with Proposition 
\ref{prop:DRn}\ref{prop:DRni2}.
Steps \Pisymbol{pzd}{194} and \Pisymbol{pzd}{197} correspond to a forward-backward iteration. Let $m_n < M_n$ be the iteration number where the Douglas-Rachford algorithm stops. The error terms involved in Step \Pisymbol{pzd}{197}
are $a_n = z_{n,m_n+\frac{1}{2}}-\prox_{\ic+\gamma_n f} x_n$ and $b_n=0$.
The properties of the algorithm are then the following:
\begin{proposition}\label{p:convDRfw}
Suppose that Problem \ref{pb:minimisation} has a solution and one of the
Assumptions \ref{a:ddd}\ref{a:ddd1}, \ref{a:ddd}\ref{a:ddd2} or
\ref{a:ddd}\ref{a:ddd3} holds.
\begin{enumerate}
\item \label{p:convDRfwi} There exists a sequence of positive integers $(\overline{M}_n)_{n\in \NN}$
such that, if $(\forall n \in \NN)$ $M_n \ge \overline{M}_n$
then, $(x_n)_{n\in \NN}$ converges weakly to a solution to Problem \ref{pb:minimisation}.
\item \label{p:convDRfwii} The sequence $(x_n)_{n\in \NN}$ lies in $C$.
\end{enumerate}
\end{proposition}
\begin{proof}
\noindent\ref{p:convDRfwi}:
Set $\rho \in ]0,1[$.
Let $n\in \NN$ and 
$(z_{n,m})_{m\in \NN}$ be defined by iterating Steps~\Pisymbol{pzd}{196}a), b)
and c). By invoking Proposition \ref{prop:DRn}\ref{prop:DRni}, we know that
$(z_{n,m+\frac{1}{2}})_{m\in \NN}$ converges strongly to $\prox_{\ic+\gamma_n f} x_n'$.
This implies that there exists $\overline{M}_n\ge 1$ such that
\begin{equation*}
(\forall m \in \NN)\qquad m \ge \overline{M}_n -1 \quad\Rightarrow\quad 
\|z_{n,m+\frac{1}{2}} - \prox_{\ic+\gamma_n f} x_n'\|
\le \rho^n.
\end{equation*}
If $M_n \ge \overline{M}_n$, we deduce that 
\begin{equation*}
\|a_n\|= 
\|z_{n,m_n+\frac{1}{2}} - \prox_{\ic+\gamma_n f} x_n'\|
\le \rho^n
\end{equation*}
since either $m_n = M_n-1$ or the algorithm stops in Step \Pisymbol{pzd}{196}d)
(in which case $z_{n,m_n}$ is a fixed point of the recursion in Step
\Pisymbol{pzd}{196}c) and
$z_{n,m_n+\frac{1}{2}} = \prox_{\ic+\gamma_n f} x_n'$).
We therefore have $\sum_{n\in \NN} \|a_n\| < \pinf$ and the conditions
for the weak convergence of the forward-backward algorithm are fulfilled.\\
\ref{p:convDRfwii}: We have chosen $x_0$ in $C$. In addition, $(\forall n \in \NN)$
$(z_{n,m+\frac{1}{2}})_m$ lies in $C$ and $x_{n+1}$ is convex combination
of $x_n$ and $z_{n,m+\frac{1}{2}}$. Hence, it is easily shown by induction
that $(\forall n \ge 1)$ $x_n \in C$.
\end{proof}

Proposition \ref{p:convDRfw}\ref{p:convDRfwi} guarantees
that, by choosing $M_n$ large enough, the algorithm allows us to
solve Problem~\ref{pb:minimisation}.
Although this result may appear somehow imprecise regarding the practical choice of 
$M_n$, it was observed in our simulations
that small values of $M_n$ are sufficient to ensure the convergence.

In addition, as a direct consequence of  Proposition \ref{p:convDRfw}\ref{p:convDRfwii}, in Step \Pisymbol{pzd}{194}, the gradient of $g$ is only evaluated on $C$.
This means that, similarly to Algorithm \ref{algo:main1}, this algorithm
is able to solve Problem \ref{prob:minimisationgen}. In the next section, we will show
that a number of image restoration problems can be formulated as Problem \ref{prob:minimisationgen}.

\section{Application to a class of image restoration problems}\label{sec:application}

\subsection{Context} \label{se:applicont}
We aim at restoring an image $\overline{y}$ in a real separable Hilbert space $\GG$ 
from a degraded observation $z\in \GG$. Here, digital images of size $N_1\times N_2$ are considered and thus $\GG = \RR^{N}$ with $N = N_1 N_2$. 
Let $T$ be a linear operator from $\GG$ to $\GG$ modelling
a linear degradation process, e.g. a convolutive blur.
The image 
$\overline{u} = T\overline{y}$ (resp. $z = (z^{(i)})_{1 \le i \le N}$) is a realization of a real-valued random vector $\overline{U}=(\overline{U}^{(i)})_{1\le i \le N}$ (resp. $Z
= (Z^{(i)})_{1\le i \le N}$). The image $\overline{U}$ is contaminated by noise.
Conditionally to $\overline{U} = (u^{(i)})_{1 \le i \le N} \in \GG$,
the random vector $Z$ is assumed to have
independent components, which are either discrete with conditional probability mass functions 
$(\mu_{Z_i \mid \overline{U}^{(i)} = u^{(i)}})_{1 \le i \le N}$, or absolutely continuous
with conditional probability density functions which are also denoted by $(\mu_{Z_i \mid \overline{U}^{(i)} = u^{(i)}})_{1 \le i \le N}$. In this paper, we are interested in probability distributions such that:
\begin{equation}
(\forall i \in \{1,\ldots,N\})(\forall \upsilon \in \RR) \qquad
\mu_{Z^{(i)} \mid \overline{U}^{(i)}=\upsilon}(z^{(i)})\propto \exp\big(-\psi_i(\upsilon)\big)
\label{eq:defpsii}
\end{equation}
where the functions $(\psi_i)_{1\le i \le N}$ take their values
in $]\minf,\pinf]$ and satisfy the following assumption.
\begin{assumption}\label{as:psi}
There exists a nonempty subset $\mathbb{I}$ of $\{1,\ldots,N\}$ 
and a constant $\delta \in \RR$ such that,
for all $i\in \{1,\ldots,N\}$,
\begin{enumerate}
\item \label{as:psii} $\dom \psi_i = ]\delta,\pinf[$ if $i\in \mathbb{I}$ and,
$\dom \psi_i = [\delta,\pinf[$ if $i\not \in \mathbb{I}$;
\item \label{as:psiiii} if $i\in \mathbb{I}$, then $\psi_i$ is twice continuously differentiable
on $]\delta,\pinf[$ such that $\inf_{\upsilon\in ]\delta,\pinf[} \psi_i(\upsilon) > \minf$ and
\begin{equation*}
\lim_{\substack{\upsilon \to \delta\\\upsilon > \delta}} \psi_i(\upsilon) = \pinf.
\end{equation*}
Its second-order derivative $\psi_i''$ is 
decreasing and satisfies
\begin{equation*}
\lim_{\upsilon \to \pinf} \psi_i''(\upsilon) = 0;
\end{equation*}
\item   \label{as:psiii} if $i \not\in \mathbb{I}$, then there exists $\alpha_i \in \RP$ such that
$(\forall \upsilon \in [\delta,\pinf[)$ $\psi_i(\upsilon) = \alpha_i \upsilon$.
\end{enumerate}
\end{assumption}
From Assumptions \ref{as:psi}\ref{as:psiiii} and \ref{as:psiii}, it is clear that the functions $(\psi_i)_{1 \le i \le  N}$
are convex (since $(\forall i \in \mathbb{I})$ $(\forall \upsilon \in ]\delta,\pinf[)$ $\psi''_i(\upsilon) \ge 0$) such that
\begin{equation}
\lim_{\substack{\upsilon \to \delta\\\upsilon > \delta}} \psi_i''(\upsilon) = \pinf
\label{eq:psiid}
\end{equation}
and they are lower semicontinuous
(since $(\forall i \in \{1,\ldots,N\})$
$\lim \inf_{\upsilon \to \delta} \psi_i(\upsilon) \ge \psi_i(\delta)$).
Examples of such functions will be provided in Sections
\ref{sec:gaussian} and \ref{sec:poisson}.

In addition, a both simple and efficient prior probabilistic model on the unknown image $\overline{y}$ is adopted by using a representation of this image in a frame \cite{Daubechies_I_1992_book_ten_lw,Han_D_2000_book_frames_bgr}. 
The frame coefficient space is the Euclidean space $\HH =\RR^K$  ($K \ge N$).  
We thus use a linear representation of the form:
\begin{equation*}
\overline{y} = F^* \overline{x}
\end{equation*}
where $F^*\,:\; \HH  \to \GG$ is a frame synthesis operator, i.e. $\underline{\nu}\, \Id \le F^* \circ F \le \overline{\nu}\, \Id$ with
$(\underline{\nu},\overline{\nu}) \in \RPP^2$ (which implies that $F^*$ is surjective).\footnote{The existence of the lower bound implies the existence of the upper bound in finite dimensional case.}
We then assume that the vector  $\overline{x}$ 
of frame coefficients 
is a realization of a random vector $\overline{X}$ with independent components.
Each component $\overline{X}^{(k)}$ with $k \in \{1,\ldots,K\}$ of $\overline{X}$,
has a probability density
$\exp(-\phi_k(\cdot))/\int_{-\infty}^{+\infty}\exp(-\phi_k(\eta))\,d\eta $ where $\phi_k$ is a finite function in $\Gamma_0(\RR)$.

Finally, we assume that we have prior information on $\overline{x}$ which can be expressed by the fact that $\overline{x}$ belongs to a closed convex set $C$ of $\HH$. The constraint set $C$ will be assumed to satisfy:
\begin{equation}
 (TC^*) \cap \dom\Psi \neq \emp
\label{eq:TCet}
\end{equation}
where 
\begin{equation*}
C^* = F^*C = \menge{F^*x}{x\in C}
\label{eq:defCet}
\end{equation*}
and
\begin{equation*}
\left(\forall u = \big(u^{(i)}\big)_{1 \le i \le N} \in \GG\right)\qquad
\Psi(u) = \sum_{i=1}^{N}  \psi_i\big(u^{(i)}\big).
\end{equation*}

With these assumptions, it can be shown (see \cite{Chaux_C_2007_ip_variational_ffbip}) that a Maximum A Posteriori (MAP) estimate of the vector of frame coefficients $\overline{x}$ can be obtained from 
$z=\big(z^{(i)}\big)_{1 \le i \le N}$ by minimizing in the Hilbert space $\HH $ the function $f + g + \ic$
where
\begin{equation} \label{eq:deff}
\left(\forall x = \big(x^{(k)}\big)_{1 \le k \le K} \in \HH\right)\qquad
f(x) = \sum_{k=1}^K \phi_k\big(x^{(k)}\big)
\end{equation}
and 
\begin{equation}\label{eq:defg}
g = \Psi \circ T \circ F^*.  
\end{equation}

We consequently have:
\begin{proposition}\label{p:MAP}
Let $\HH = \RR^K$ and $\GG = \RR^N$ with $K\ge N$. Let $f$ and $g$ be defined by \eqref{eq:deff} and \eqref{eq:defg}, respectively, where $T\colon \GG \to \GG$ is a linear operator.
Under Assumption \ref{as:psi} and Condition \eqref{eq:TCet}, then
\begin{enumerate}
\item \label{p:MAPi} $f$ and $g$ are in $\Gamma_0(\HH)$;
\item  \label{p:MAPii} if $f$ is coercive\footnote{This means that $\lim_{\|x\| \to \pinf} f(x) = \pinf$.} or $\dom g\,\cap\, C$ is bounded, then the minimization of $f+g+\ic$
admits a solution. 
In addition, if $f$ is strictly convex on $\dom g \,\cap\, C$,
the solution is unique.
\end{enumerate}
\end{proposition}
\begin{proof}
\noindent\ref{p:MAPi}: It is clear that $f$  is a finite convex function of $\HH$.
As the 
functions $(\psi_i)_{1 \le i \le  N}$
are in $\Gamma_0(\RR)$,
$\Psi$ belongs to $\Gamma_0(\GG)$. In addition,
by using \eqref{eq:TCet}, we have $\ran(T\circ F^*) \cap \dom \Psi \neq \emp$.
This allows us to deduce that
$\dom g \neq \emp$
and, therefore, $g \in \Gamma_0(\HH)$.\\
\ref{p:MAPii}: We have $\dom f\, \cap\, \dom g\, \cap\, C \neq \emp$ since
$\dom f = \HH$ and \eqref{eq:TCet} shows that $\dom g\, \cap\, C 
= \dom(\Psi \circ T \circ F^*)\, \cap\, C \neq \emp$.
Since $f$ and $g$ 
are in $\Gamma_0(\HH)$, we deduce that $f+g+\ic$
is in $\Gamma_0(\HH)$. \\
Suppose now that $f$ is coercive. By Assumption \ref{as:psi}\ref{as:psiiii}, $(\forall i\in \mathbb{I})$
$\inf_{\upsilon \in ]\delta,\pinf[} \psi_i(\upsilon) > \minf$ whereas,
due to Assumption \ref{as:psi}\ref{as:psiii},
$(\forall i\not\in\mathbb{I})$ $\inf_{\upsilon \in [\delta,\pinf[} \psi_i(\upsilon)
= \alpha_i \delta$. This implies that $\inf \Psi(\GG) > \minf$ and, consequently, $\inf g(\HH) \ge \inf \Psi(\GG) > \minf$. As a result, $f+g+\ic \ge f+\ic+\inf g(\HH)$ is coercive.
When $\dom g \,\cap\, C$ is bounded, $f+g+\ic$ also is coercive.
The existence of a solution to the minimization problem follows from classical results in convex analysis \cite[Chap. 3, Prop. 1.2]{Ekeland_I_1999_book_Convex_aavp}.\\
When $f$ is strictly convex on $\dom g\, \cap\, C$, 
the uniqueness of the solution follows from the fact that $f+g+\ic$ is strictly convex
\cite[Chap. 3, Prop. 1.2]{Ekeland_I_1999_book_Convex_aavp}.
\end{proof}
\begin{remark} \label{re:fcoerc}
The function $f$ is coercive (resp. strictly convex) if and only if 
the functions $(\phi_k)_{1 \le k \le N}$ are coercive {\rm \cite[Prop. 3.3(iii)(c)]{Chaux_C_2007_ip_variational_ffbip}} (resp. strictly convex).
\end{remark}

\subsection{Quadratic extension}

If we now investigate the Lipschitz-continuity of the gradient of $g$, it turns out
that this property may be violated since $\Psi$ is not finite. Due to \eqref{eq:psiid}, the gradient of
$g$ is not even guaranteed to be Lipschitz-continuous on $\inte \dom g$.

To circumvent this problem, it can be noticed that, because of Assumption \ref{as:psi}\ref{as:psiiii} and \eqref{eq:psiid}, for all $i \in \mathbb{I}$,
there exists a decreasing function $\upsilon_i\,:\,\RPP \to ]\delta,\pinf[$
such that $\lim_{\theta \to \pinf} \upsilon_i(\theta) = \delta$ and
\begin{equation}
(\forall \theta\in \RPP)
(\forall \upsilon \in ]\delta,\pinf[)\qquad 0 \le \psi_i''(\upsilon) \le \theta
\Leftrightarrow \upsilon \ge \upsilon_i(\theta).
\label{eq:defuit}
\end{equation}
Let us now consider  the function $g_\theta = \Psi_\theta \circ T \circ F^*$ with
$\theta \in \RPP$, where
\begin{equation*}
\left(\forall u = \big(u^{(i)}\big)_{1 \le i \le N} \in \GG\right)\qquad
\Psi_\theta(u) = \sum_{i=1}^{N}  \psi_{\theta,i}\big(u^{(i)}\big)
\end{equation*}
and the functions $(\psi_{\theta,i})_{1\le i \le N}$ are chosen such that,
\begin{equation}
(\forall \upsilon \in \RR)\qquad \psi_{\theta,i}(\upsilon) =
\begin{cases}
\displaystyle \frac{\theta}{2} \upsilon^2+\zeta_{i,1}(\theta)\; \upsilon+\zeta_{i,0}(\theta)
& \mbox{if $i\in \mathbb{I}$ and 
$\delta-\epsilon(\theta) \le \upsilon < \upsilon_{i}(\theta)$}\\
\alpha_i \upsilon & \mbox{if $i \not \in \mathbb{I}$
and $\delta-\epsilon(\theta) \le \upsilon < \delta$}\\
\psi_i(\upsilon) & \mbox{otherwise.}
\end{cases}
\label{eq:defpsiit}
\end{equation}
Hereabove, $\epsilon\,:\,\RPP \to\RPP$ is a decreasing function and,
\begin{align*}
(\forall i \in \mathbb{I})\qquad
\zeta_{i,0}(\theta) & = \psi_i\big( \upsilon_{i}(\theta)\big)-\upsilon_{i}(\theta) \psi'_{i}\big( \upsilon_{i}(\theta)\big)+\frac{\theta}{2}\big(\upsilon_{i}(\theta)\big)^2
\\
\zeta_{i,1}(\theta) & = \psi'_{i}\big( \upsilon_{i}(\theta)\big)-\theta  \upsilon_{i}(\theta).
\end{align*}
For every $i\in \mathbb{I}$, the constants $\zeta_{i,0}(\theta)$ and $\zeta_{i,1}(\theta)$ have been determined so as to guarantee the continuity of $\psi_{\theta,i}$ and of its first order derivative 
at $\upsilon_i(\theta)$.
Consequently, the following result can be obtained:
\begin{proposition}\label{p:gt}
Suppose that Assumption \ref{as:psi} and Condition \eqref{eq:TCet} hold. Then,
\begin{enumerate}
\item\label{p:gt0} $(\forall \theta \in \RPP)$ $g_\theta \in \Gamma_0(\HH)$.
\item\label{p:gti} $\big(\forall (\theta_1,\theta_2)\in \RPP^2\big)$,
$\theta_1 < \theta_2$ $\Rightarrow$ 
$g_{\theta_1} \le g_{\theta_2} \le g$.
\item\label{p:gtii} 
For every $\theta \in \RPP$, if $TC^* \subset ]\delta-\epsilon(\theta),\pinf[^N$, then
$g_\theta$ has a Lipschitz-continuous gradient over 
$C$
with constant $\beta_\theta = \theta \|TF^*\|^2 \le   \theta\overline{\nu}  \|T\|^2$.
\item\label{p:gtiii} For every $\theta \in \RPP$, if $f$ is coercive or if $\dom g_\theta\,\cap\, C$ is bounded, then the minimization of 
$f+g_\theta+\ic$ admits a solution. In addition, if $f$ is strictly convex
on $\dom g_\theta \cap C$, then $f+g_\theta+\ic$
has a unique minimizer $\widetilde{x}_\theta$.
\item \label{p:gtv}
Assume that 
\begin{enumerate}
\item $\lim_{\theta\to \pinf} \epsilon(\theta) = 0$,
\item \label{as:enplus} $TC^* \subset [\delta,\pinf[^N$,
\item \label{as:coercb}
$f$ is coercive or 
$C$ is bounded,
\item \label{as:strictc} $f$ is strictly convex on $C$.
\end{enumerate}
Then, there exists $\overline{\theta} \in \RPP$ such that, for every $\theta \in [\overline{\theta},\pinf[$, the minimizer $\widetilde{x}_\theta$ of $f+g_\theta+\iota_C$ is the minimizer of $f+g+\iota_C$.
\end{enumerate}
\end{proposition}
\begin{proof}
\ref{p:gt0} 
Since $\Psi_\theta$ is defined and continuous on $[\delta-\epsilon(\theta),\pinf[^N$ and,
$(\forall i \in \{1,\ldots,N\})$
$(\forall \upsilon \in ]\delta-\epsilon(\theta),\pinf[)$
$\psi_{\theta,i}''(\upsilon)\ge 0$, we have $\Psi_\theta \in \Gamma_0(\GG)$. In addition, $\dom \Psi_\theta \cap \ran(T\circ F^*) \supset \dom \Psi \cap \ran(T\circ F^*) \neq \emp$. Thus, $g_\theta \in \Gamma_0(\HH)$.\\
\ref{p:gti}
As a consequence of \eqref{eq:defpsiit} and \eqref{eq:defuit}, 
we have, for every $i\in \mathbb{I}$,
\begin{equation*}
(\forall \upsilon \in ]\delta,\upsilon_{i}({\theta_2})[)\qquad \psi_{i}''(\upsilon) > \psi_{{\theta_2},i}''(\upsilon) = {\theta_2}.
\end{equation*}
So $\psi_i'-\psi_{{\theta_2},i}'$ is a strictly increasing function
over $]\delta,\upsilon_i({\theta_2})]$ and
\begin{equation*}
(\forall \upsilon \in ]\delta,\upsilon_i({\theta_2})[)\qquad
\psi_{i}'(\upsilon) -\psi_{{\theta_2},i}'(\upsilon)
< \psi_{i}'\big(\upsilon_{i}({\theta_2})\big) - \psi_{{\theta_2},i}'\big(\upsilon_{i}({\theta_2})) = 0
\end{equation*}
which, in turn, shows that $\psi_{i}- \psi_{{\theta_2},i}$
is strictly decreasing on $]\delta,\upsilon_i({\theta_2})]$
and
\begin{equation*}
(\forall \upsilon \in ]\delta,\upsilon_i({\theta_2})[)\qquad
\psi_{i}(\upsilon) -\psi_{{\theta_2},i}(\upsilon)
> \psi_{i}\big(\upsilon_{i}({\theta_2})\big) - \psi_{{\theta_2},i}\big(\upsilon_{i}({\theta_2})) = 0.
\end{equation*}
In addition, we know that, if $(i\in \mathbb{I}$ and $\upsilon \le \delta)$ or
$(i \not\in \mathbb{I}$ and $\upsilon < \delta)$, then
$\psi_{i}(\upsilon)=\pinf$ and,
if $\big(i\in \mathbb{I}$ and $\upsilon
\ge \upsilon_i({\theta_2})\big)$ or $(i\not\in \mathbb{I}$ and $\upsilon \ge \delta)$, then
$\psi_{i}(\upsilon) = \psi_{{\theta_2},i}(\upsilon)$. We deduce that, for
all $i\in \{1,\ldots,N\}$, $\psi_i \ge \psi_{{\theta_2},i}$
and, therefore $g$ is lower bounded by $g_{\theta_2}$.\\
By proceeding similarly, we have, for every $i\in \mathbb{I}$,
\begin{alignat*}{2}
&(\forall \upsilon \in [\upsilon_i(\theta_1),\pinf[)\qquad
&&\psi_{\theta_2,i}(\upsilon)= \psi_i(\upsilon) = \psi_{\theta_1,i}(\upsilon)\\
&(\forall \upsilon \in ]\delta-\epsilon(\theta_2),\upsilon_i(\theta_1)[)\qquad
&&\psi''_{\theta_2,i}(\upsilon) > \theta_1 = \psi''_{\theta_1,i}(\upsilon)\nonumber\\
&\Rightarrow\;\; (\forall \upsilon \in ]\delta-\epsilon(\theta_2),\upsilon_i(\theta_1)[)\qquad
&& \psi'_{\theta_2,i}(\upsilon) < \psi'_{\theta_1,i}(\upsilon)
\nonumber\\
&\Rightarrow\;\;(\forall \upsilon \in [\delta-\epsilon(\theta_2),\upsilon_i(\theta_1)[)\qquad
&& \psi_{\theta_2,i}(\upsilon) > \psi_{\theta_1,i}(\upsilon).
\end{alignat*}
In addition, 
\begin{equation*}
(\forall i \in \{1,\ldots,N\})
(\forall \upsilon \in ]\minf,\delta-\epsilon(\theta_2)[)\qquad
\psi_{\theta_2,i}(\upsilon) = \pinf \ge \psi_{\theta_1,i}(\upsilon)
\end{equation*}
and
\begin{equation*}
(\forall i \not\in \mathbb{I})
(\forall \upsilon \in [\delta-\epsilon(\theta_2),\pinf[)\qquad
\psi_{\theta_2,i}(\upsilon) =\psi_{\theta_1,i}(\upsilon).
\end{equation*}
This shows that $\Psi_{\theta_2} \ge \Psi_{\theta_1}$
and, consequently, $g_{\theta_2} \ge g_{\theta_1}$.

\noindent\ref{p:gtii}:
As already mentioned, $\dom \Psi_\theta = [\delta-\epsilon(\theta),\pinf[^N$.
Consider 
\begin{equation*}
O_\theta = (TF^*)^{-1}(]\delta-\epsilon(\theta),\pinf[^N) = \menge{x\in\HH}{TF^*x \in ]\delta-\epsilon(\theta),\pinf[^N}.
 \label{eq:defO}
\end{equation*}
 $O_\theta$ is an open set and, 
as $TC^* \subset ]\delta-\epsilon(\theta),\pinf[^N$, we have: $C \subset O_\theta$.
In addition, 
the function $g_\theta$ is differentiable on $O_\theta$ 
and its gradient is \cite[Chap. 1, Prop. 5.7]{Ekeland_I_1999_book_Convex_aavp}
\begin{equation}
(\forall x \in O_\theta)
\qquad \nabla g_\theta(x) = FT^*\big(\nabla \Psi_\theta(TF^*x)\big)
\label{eq:gradcomp}
\end{equation} 
where
\begin{equation*}
\big(\forall u =(u^{(i)})_{1\le i \le n} \in 
]\delta-\epsilon(\theta),\pinf[^N\big)
\qquad \nabla\Psi_\theta(u) = \big(\psi'_{\theta,i}(u^{(i)})\big)_{1 \le i \le N}.
\end{equation*}
We have then
\begin{multline*}
\big(\forall u =(u^{(i)})_{1\le i \le n} \in 
]\delta-\epsilon(\theta),\pinf[^N\big)
\big(\forall v =(v^{(i)})_{1\le i \le n} \in 
]\delta-\epsilon(\theta),\pinf[^N\big)\\
\|\nabla\Psi_\theta(u)-\nabla\Psi_\theta(v)\|
= \Big(\sum_{i=1}^N \big(\psi'_{\theta,i}(u^{(i)})- \psi'_{\theta,i}(v^{(i)})\big)^2\Big)^{1/2}
\end{multline*}
and, by the mean value theorem,
\begin{align*}
(\forall i \in \{1,\ldots,N\})\quad
\big|\psi'_{\theta,i}(u^{(i)})- \psi'_{\theta,i}(v^{(i)})\big|
&\le |u^{(i)}-v^{(i)}| 
\sup_{\xi \in ]\delta-\epsilon(\theta),\pinf[} 
|\psi''_{\theta,i}(\xi)| \nonumber\\
&\le \theta |u^{(i)}-v^{(i)}|.
\end{align*}
This yields
\begin{equation*}
\big(\forall u\in 
]\delta-\epsilon(\theta),\pinf[^N\big)
\big(\forall v\in 
]\delta-\epsilon(\theta),\pinf[^N\big)
\qquad\|\nabla\Psi_\theta(u)-\nabla\Psi_\theta(v)\|
\le \theta \|u-v\|
\end{equation*} 
and, we deduce from \eqref{eq:gradcomp} that
\begin{equation*}
\big(\forall (x,x') \in 
O_\theta^2\big)
\qquad\|\nabla g_\theta(x)-\nabla g_\theta(x')\| \le 
\theta \|T F^*\|^2 \|x-x'\|.
\end{equation*}
and $\|T F^*\|^2 \le \|F\|^2\|T\|^2 \le \overline{\nu}\|T\|^2$.\\
\ref{p:gtiii}: The proof is similar to that of Proposition \ref{p:MAP}\ref{p:MAPii}.\\
\ref{p:gtv}: In the following, we use the notation:
$h = f+g+\ic$ and
$(\forall \theta \in \RPP)$ $h_\theta = f+g_\theta+\ic$.\\
Let $(\theta_\ell)_{\ell \in \NN}$ be an increasing sequence
of $\RPP$ such that $\lim_{\ell \to \pinf} \theta_\ell = \pinf$.
As a consequence of \ref{p:gt0} and \ref{p:gti}, $(h_{\theta_\ell})_{\ell \in \NN}$
is an increasing sequence of functions in $\Gamma_0(\HH)$.
We deduce from \cite[Proposition 7.4(d)]{Rockafellar_RT_2004_book_Variational_a}
that $(h_{\theta_\ell})_{\ell \in \NN}$ epi-converges
to its pointwise limit. By using \eqref{eq:defpsiit} in combination with the facts that $(\forall i \in \mathbb{I})$
$\lim_{\theta\to \pinf} \upsilon_i(\theta) = \delta$
and $\lim_{\theta \to \pinf} \epsilon(\theta) = 0$,
we see that the pointwise limit is equal to $h$.\\
Under Assumptions \ref{as:enplus} and \ref{as:coercb}, $(\forall \ell \in \NN)$ $h_{\theta_\ell}$
is coercive since $C\,\cap\,\dom g_{\theta_\ell} = C$.
Equivalently, its level sets 
$\lev{\eta} h_{\theta_\ell} =\menge{x \in \HH}{h_{\theta_\ell}(x) \le \eta}$
with $\eta \in \RR$, are bounded.
 $(h_{\theta_\ell})_{\ell \in \NN}$
being a sequence of increasing functions, 
$\cup_{\ell\in\NN} \lev{\eta}h_{\theta_\ell} = \lev{\eta}h_{\theta_0}$
is bounded.
As the functions $h_{\theta_\ell}$ with $\ell \in \NN$ and $h$ are lower semicontinuous and proper, \cite[Theorem 7.33]{Rockafellar_RT_2004_book_Variational_a} allows us to claim that the sequence $(\widetilde{x}_{\theta_\ell})_{\ell \in \NN}$
converges to the minimizer $\widetilde{x}$ of $h$ 
(by Assumptions~\ref{as:enplus} and \ref{as:strictc}, both $h_{\theta_\ell}$ with $\ell \in \NN$ and
$h$ have a unique minimizer due to the strict convexity of $f$
on $(C\,\cap\,\dom g)\subset (C\,\cap\,\dom g_{\theta_\ell}) = C$
and, Propositions \ref{p:MAP}\ref{p:MAPii} and \ref{p:gt}\ref{p:gtiii}).
As $\widetilde{x} \in \dom h$, $(\forall i \in \mathbb{I})$
$(TF^*\widetilde{x})^{(i)} \in \dom \psi_i = ]\delta,\pinf[$,
where, for every $x\in \HH$ and $i\in \{1,\ldots,N\}$,
$(TF^* x)^{(i)}$ denotes the $i$-th component of vector $TF^* x$.
Since $\lim_{\ell \to \pinf} \widetilde{x}_{\theta_\ell} =
\widetilde{x}$, we have, for every $i\in \mathbb{I}$,
\begin{align*}
(\forall \eta \in \RPP) (\exists \ell_{\eta,i} \in \NN)
\text{ such that }\\ 
\begin{split}
(\forall \ell \in \NN)\quad \ell \ge \ell_{\eta,i}  \Rightarrow &  
| (TF^*\widetilde{x}_{\theta_\ell})^{(i)} - (TF^*\widetilde{x})^{(i)}|
< \eta\\
\Rightarrow &  
(TF^*\widetilde{x}_{\theta_\ell})^{(i)} > \min_{i\in \mathbb{I}} (TF^*\widetilde{x})^{(i)} - \eta.
\end{split}
\end{align*}
By setting $\displaystyle \eta = \frac{\min_{i\in \mathbb{I}}(TF^*\widetilde{x})^{(i)}-\delta}{2} > 0$ and $\ell_\eta = \max_{i\in\mathbb{I}} \ell_{\eta,i}$, we deduce that
\begin{equation}
(\forall \ell \in \NN)\quad \ell \ge \ell_\eta 
\Rightarrow (TF^*\widetilde{x}_{\theta_\ell})^{(i)} \ge \underline{\upsilon} 
\label{eq:minlbound}
\end{equation}
where $\displaystyle \underline{\upsilon} = \frac{\delta+\min_{i\in \mathbb{I}} (TF^*\widetilde{x})^{(i)}}{2}> \delta$.
In addition, since $\lim_{\ell \to \pinf} \theta_\ell = \pinf$ $\Rightarrow$
$\lim_{\ell \to \pinf} 
\max_{i\in \mathbb{I}} \upsilon_{i}(\theta_\ell) = \delta$, there exists 
$\overline{\ell}\ge \ell_\eta$ such that $(\forall i \in\mathbb{I})$
$\upsilon_{i}(\theta_{\overline{\ell}}) \le \underline{\upsilon}$.
By using \eqref{eq:defpsiit}, this implies that $(\forall i \in \mathbb{I})$ $(\forall \upsilon \in 
[\underline{\upsilon},\pinf[)$, $\psi_{\theta_{\overline{\ell}},i}(\upsilon) = 
\psi_i(\upsilon)$. By defining now
\begin{equation*}
D =
\menge{x \in \dom g}
{(\forall i \in \mathbb{I})\; (TF^* x)^{(i)} \in [\underline{\upsilon},\pinf[}
\end{equation*}
we deduce that
$(\forall x \in D)$  $h_{\theta_{\overline{\ell}}}(x) = h(x)$. 
Moreover, according to Assumption~\ref{as:enplus}, for every $\ell \in \NN$,
if $i\not\in \mathbb{I}$,
\begin{equation}
(TF^*\widetilde{x}_{\theta_\ell})^{(i)}\in [\delta,\pinf[.
\label{eq:audessus}
\end{equation}
Altogether, \eqref{eq:minlbound} and \eqref{eq:audessus} show that both 
$\widetilde{x}_{\theta_{\overline{\ell}}}$ and $\widetilde{x}$
belong to $D$. Consequently, as $\widetilde{x}_{\theta_{\overline{\ell}}}
= \arg\min_{x\in \HH} h_{\theta_{\overline{\ell}}}(x)$, we have:
$h(\widetilde{x}_{\theta_{\overline{\ell}}})=
h_{\theta_{\overline{\ell}}}(\widetilde{x}_{\theta_{\overline{\ell}}})
\le h_{\theta_{\overline{\ell}}}(\widetilde{x}) = h(\widetilde{x})$, which
proves that $\widetilde{x}_{\theta_{\overline{\ell}}}=\widetilde{x}$.\\
Considering now $\theta \in [\theta_{\overline{\ell}},\pinf[$,
from \ref{p:gti} we get:
$h_{\theta_{\overline{\ell}}}\le 
h_\theta \le h$.
Thus, $h(\widetilde{x})=h_{\theta_{\overline{\ell}}}(\widetilde{x}) \le h_\theta(\widetilde{x}) 
\le h(\widetilde{x})$,
which results in $h_\theta(\widetilde{x})= h(\widetilde{x})$, while
\begin{equation*}
(\forall x \in \HH)\qquad 
h_\theta(x) \ge 
h_{\theta_{\overline{\ell}}}(x)
\ge h_{\theta_{\overline{\ell}}}(\widetilde{x})=  
h(\widetilde{x}).
\end{equation*}
This allows us to conclude that $\widetilde{x}_\theta = \widetilde{x}$
as soon as $\theta \ge \theta_{\overline{\ell}} = \overline{\theta}$.
\end{proof}
\begin{remark}
\begin{enumerate}
\item A polynomial approximation of the objective function was considered in \cite{Fessler_JA_1995_tip_Hybrid_ppofftirfts} which is different from the proposed quadratic extension technique.
\item As expressed by Proposition \ref{p:gt}\ref{p:gti}, $g_\theta$
(resp. $f+g_\theta+\ic$) with $\theta > 0$
constitutes a lower approximation of $g$ (resp. $f+g+\ic$),
which becomes closer as $\theta$ increases.
\item  As shown by Proposition \ref{p:gt}\ref{p:gtii}, the main role of parameter $\theta$ is to control the Lipschitz constant of the gradient of this approximation of $g$.
\item At the same time, Proposition \ref{p:gt}\ref{p:gtv} indicates that this parameter allows us to
control the closeness of the approximation to a minimizer of the original MAP criterion. This approximation becomes perfect when $\theta$ becomes greater than some
value $\overline{\theta}$.
\end{enumerate}
\end{remark}
Under the assumptions of Proposition \ref{p:gt}\ref{p:gtii},
the minimization of $f+g_\theta+\ic$ with $\theta \in \RPP$ is a problem
of the type of Problem \ref{prob:minimisationgen}. Therefore, Propositions~\ref{p:convfwDR} and \ref{p:convDRfw} show that, 
provided that $f$ is coercive or $C$ is bounded, 
Algorithms~\ref{algo:main1}
and \ref{algo:main} can be applied in this context. 
In addition, Proposition \ref{p:gt}\ref{p:gtv} suggests that, by choosing $\theta$ large enough, a solution to the original MAP criterion can be found.
However, according to Proposition~\ref{p:gt}\ref{p:gtii}, a large value of $\theta$ induces a large value of the Lipschitz constant $\beta_\theta$. This means that a small value of the step-size parameter must also to be used in the forward iteration of the algorithms, which is detrimental to the convergence speed. 
In practice, the choice of $\theta$ results from a trade-off as will be illustrated by the numerical results.

\subsection{First example}\label{sec:gaussian}
\subsubsection{Model}
We want to restore an image $\overline{y}\in\RP^N$ corrupted by a linear operator $T\,:\,\GG \to \GG$ and an additive noise $w\in \GG$, having the observation 
\begin{equation*}
z=T\overline{y}+w=\overline{u}+w.
\end{equation*}
In addition, the linear operator $T$ is assumed to be nonnegative-valued (in the sense that
the matrix associated to $T$ has nonnegative elements)
and, $w = (w^{(i)})_{1 \le i \le N}$ is a realization of an independent zero-mean Gaussian noise $W = (W^{(i)})_{1 \le i \le N}$. The variance of each random variable $W^{(i)}$ with $i\in \{1,\ldots,N\}$ is signal-dependent and is equal to $\sigma_i^2(\overline{u}^{(i)})$ where
\begin{equation*}
(\forall \upsilon \in 
[0,+\infty[)\qquad 
\sigma_i^2(\upsilon) = \frac{\upsilon}{2\alpha_i}
\end{equation*}
with  $\alpha_i \in \RPP$.
So, the functions $(\psi_i)_{1\le i \le N}$ as defined in \eqref{eq:defpsii}
are, when $z^{(i)} \neq 0$,
\begin{equation*}
(\forall \upsilon \in \RR)\qquad \psi_i(\upsilon) =
\begin{cases}
\displaystyle \frac{\alpha_i\big(\upsilon-z^{(i)}\big)^2}{\upsilon}
& \mbox{if $\upsilon \in \RPP$}\\
+\infty & \mbox{otherwise}
\end{cases}
\end{equation*}
and, when $z^{(i)} = 0$,
\begin{equation*}
(\forall \upsilon \in \RR)\qquad \psi_i(\upsilon) =
\begin{cases}
\displaystyle \alpha_i\upsilon
& \mbox{if $\upsilon \in \RP$}\\
+\infty & \mbox{otherwise.}
\end{cases}
\end{equation*}
So, provided that $z\neq 0$,
Assumption \ref{as:psi} is satisfied with $\delta = 0$
and $\mathbb{I} =$\linebreak$ \menge{i\in \{1,\ldots,N\}}{z^{(i)} \neq 0}$
since, for all $i\in \mathbb{I}$, 
\begin{align*}
(\forall \upsilon \in \RPP)\qquad \psi_i'(\upsilon) &=\alpha_i 
\frac{\upsilon^2-(z^{(i)})^2}{\upsilon^2}\\
\psi_i''(\upsilon) & =  \frac{2\alpha_i(z^{(i)})^2}{\upsilon^3}.
\end{align*}
We deduce from \eqref{eq:defuit} that, for every
$i \in \mathbb{I}$,
\begin{equation*}
(\forall \theta \in \RPP)\qquad
\upsilon_i(\theta) = \Big(\frac{2\alpha_i(z^{(i)})^2}{\theta}\Big)^{1/3}.
\end{equation*}

\subsubsection{Simulation results}
\label{sec:simulSigDep}
Here, $T$ is either a $3\times 3$ or a $7\times 7$ uniform convolutive blur with $\|T\|=1$. The $512 \times 512$ satellite image $\overline{y}$ 
($N = 512^2$) shown in Fig.~\ref{fig:marseille_sigdep}(a) has been degraded by $T$ and a signal-dependent additive noise following the model described in the previous section with $\alpha_i \equiv 1$ or $\alpha_i \equiv 5$.
The degraded image $z$ displayed in Fig.~\ref{fig:marseille_sigdep}(b) corresponds to a $7\times 7$ uniform blur and $\alpha_i \equiv 1$.

A twice redundant dual-tree tight frame representation \cite{Chaux_C_2006_tip_ima_adtmbwt} ($\underline{\nu}=\overline{\nu}=2$, $K = 2 N$)  
using symlet filters of length $6$ \cite{Daubechies_I_1992_book_ten_lw} has been employed in this example.
The potential functions $\phi_k$ are taken of the form $\chi_k |\,.\,|+\omega_k |\,.\,|^{p_k}$ where $(\chi_k,\omega_k)\in \RPP^2$ and $p_k \in \{4/3,3/2,2\}$ are subband adaptive. These parameters have been determined by a maximum likelihood approach.
The function $f$ as defined by \eqref{eq:deff} is therefore coercive and strictly convex (see Remark~\ref{re:fcoerc}). 

A constraint on the solution is introduced to take into account the range of admissible values in the image by choosing
\begin{eqnarray}
C^* = [0,255]^N.
\label{eq:constval}
\end{eqnarray}
Due to the form of the operator $T$, $TC^* = C^*$ and Condition \eqref{eq:TCet} is therefore
satisfied. Proposition \ref{p:MAP} thus guarantees that a unique solution $\widetilde{x}$ to the 
MAP estimation problem exists. According to Proposition \ref{p:gt}\ref{p:gtiii},
for every $\theta \in \RPP$,
a unique minimizer $\widetilde{x}_\theta$ of $f+g_\theta+\iota_C$ also exists which allows us to approximate
$\widetilde{x}$ as stated by Proposition \ref{p:gt}\ref{p:gtv}.

Since, for every $\theta \in \RPP$, $TC^* = C^* \subset [-\epsilon(\theta),\pinf[^N$,
Proposition \ref{p:gt}\ref{p:gtii} shows that $g_\theta$ has a Lipschitz-continuous
gradient over $C$ and Algorithms \ref{algo:main1} and \ref{algo:main} can be used to compute $\widetilde{x}_\theta$.
The two algorithms are subsequently tested. 

On the one hand, when Algorithm \ref{algo:main1} is used,
the initialization is performed by setting $z_0 = P_C z$ and we choose $\gammat \equiv 60$ and $\lambdat_m \equiv 1$. 
The projection onto $C=(F^*)^{-1} C^*$ is $P_C = \prox_{\iota_{C^*}\circ F^*}$ which can be computed by using Proposition~\ref{p:linprox}
with $L = F^*$.
The other parameters have been fixed to $\lambda_{m,n} \equiv 1$ and $\gamma_{m,n} \equiv 0.995/(\kappa\theta)$, 
in compliance with Proposition \ref{p:gt}\ref{p:gtii}. 
The convergence of the algorithm is secured by Proposition~\ref{p:convfwDR}
since Assumption \ref{a:ddd}\ref{a:ddd1} trivially holds.
However, to improve the convergence profile, the following empirical rule
for choosing the number $N_m$ of forward-backward iterations
has been substituted for the necessary Conditions \eqref{eq:condconvfwDR0} and \eqref{eq:condconvfwDRm}:
\begin{equation}
N_m = \inf\menge{n\in \NN^*}{\|x_{m,n}-x_{m,n-1}\| \le \eta}
\label{eq:Nmchoice}
\end{equation}
with $\eta = 10^{-4}$.

On the other hand, when Algorithm~\ref{algo:main} is used,
the parameters have been chosen as follows : $\lambda_n \equiv 1$, $\tau_{n,m} \equiv 1$ and $\gamma_n \equiv 0.995/\theta$. The algorithm has been initialized by setting $x_0 = P_C z$ where the projection onto $C$ is computed as described previously.
The convergence of the algorithm is ensured by Proposition~\ref{p:convDRfw}.
The number $M_n$ of Douglas-Rachford iterations has been fixed as follows:
\begin{equation}
M_n = \inf\menge{m\in \NN^*}{\|z_{n,m}-z_{n,m-1}\| \le \eta}
\end{equation}
with the same value of $\eta$ as for the first algorithm.

The error between an image $y$ and the original image $\overline{y}$ is evaluated by the signal to noise ratio (SNR) defined as
$20 \log_{10}(\|\overline{y}\|/\|y-\overline{y}\|)$.

Three objectives are targeted in our experiments. First, 
we want to study the performance of the proposed approach, using the redundant dual-tree transform (DTT). 
The results presented in Tab.~\ref{tab:snrmars_sigdep} have been generated by
Algorithm~\ref{algo:main1}, but Algorithm~\ref{algo:main} leads to the same results.

\begin{table}[htbp]
\centering
\begin{tabular}{|c||c|c|c|c|c||c|c|c|c|}
\cline{3-10}
\multicolumn{2}{c|}{}& \multicolumn{4}{c||}{$3 \times 3$ blur}& \multicolumn{4}{c|}{$7 \times 7$ blur}\\
\hline
 & $\theta$ & 0.025 & 0.05 & 5  & 7 & 0.025 & 0.05 & 5 & 7\\ 
  \cline{2-10}
 $\alpha_i=1$ & SNR & 13.9 & 16.3 & \textbf{16.8} & 16.8 &  10.9 & 11.9 & \textbf{12.1} & 12.1\\
\hline
\hline
 & $\theta$ & 0.15 & 0.25 & 10 & 12 & 0.15 & 0.25 & 10 & 12\\
 \cline{2-10}
 $\alpha_i=5$ & SNR & 15.9 & 18.0 & \textbf{18.8} & 18.8 & 12.6 & 13.3 & \textbf{13.7} & 13.7 \\
\hline
\end{tabular}
\caption{\textrm{SNR} for the satellite image.\label{tab:snrmars_sigdep}}
\end{table}

As suggested by
Proposition \ref{p:gt}\ref{p:gtv}, as $\theta$ increases, the 
image is better restored. 
The effectiveness of the proposed approach is also demonstrated visually in Fig.~\ref{fig:marseille_sigdep}(c) showing the
restored image 
when $T$ is a $7 \times 7$ uniform blur, $\alpha_i\equiv 1$ and $\theta=0.05$.
It can be observed that the algorithm allows us to recover most of the 
details which were not perceptible due to blur and noise. 
\begin{figure}[htbp]
\begin{center}
\begin{tabular}{cc}
\includegraphics[width=6cm]{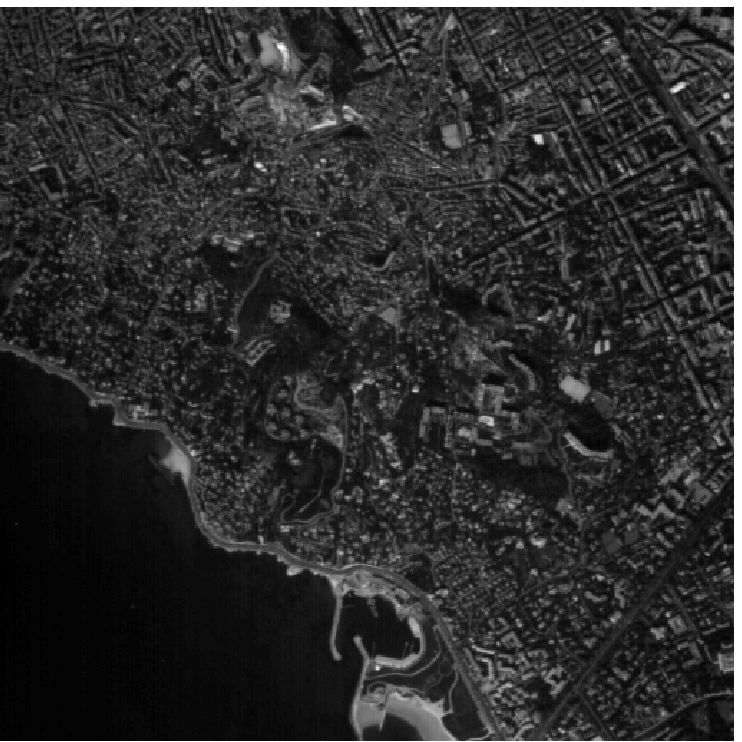} & \includegraphics[width=6cm]{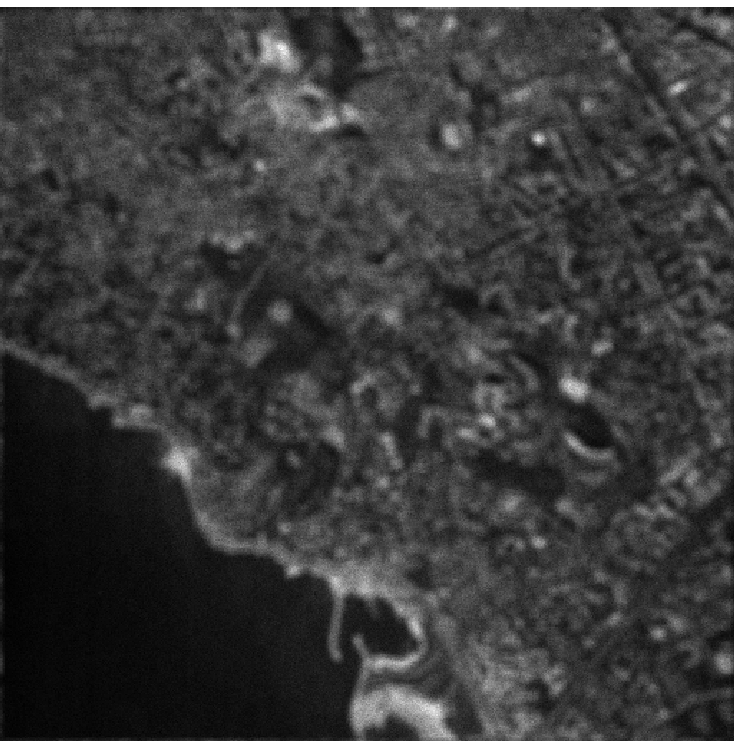}\\
(a) & (b) \\
\multicolumn{2}{c}{\includegraphics[width=6cm]{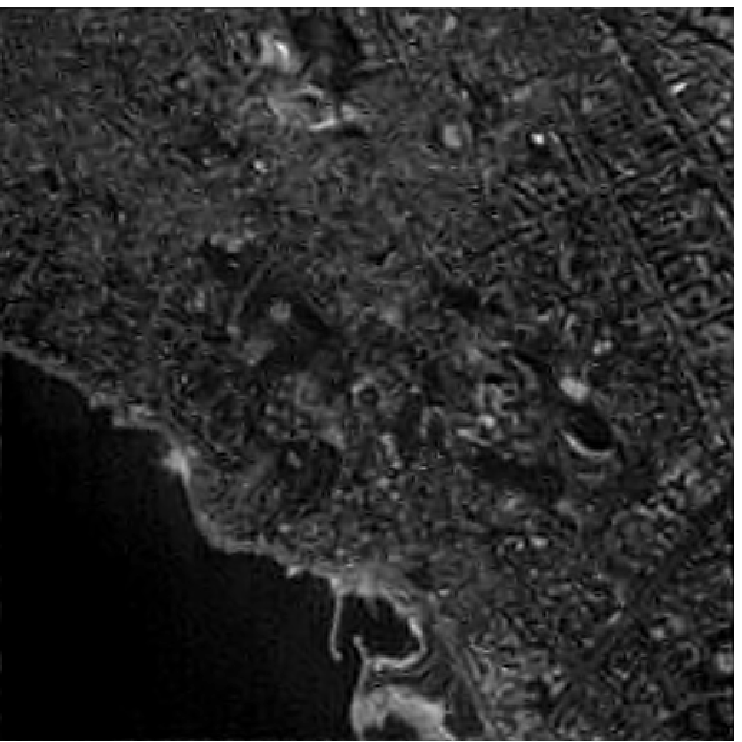}} \\
\multicolumn{2}{c}{(c)}
\end{tabular}
\end{center}
\caption{Results for a satellite image of the city of Marseille. (a) Original image, (b) degraded image, (c) restored using a DTT. \label{fig:marseille_sigdep}}
\end{figure}

Secondly, we aim at comparing the two proposed algorithms in terms of convergence for a given value of $\theta$. In Fig. \ref{fig:compalgo}, 
the MAP criterion value is plotted as a function of the computational time
for a $7\times 7$ blur, 
$\alpha_i\equiv 5$ and $\theta=0.25$.
For improved readibility, the criterion has been normalized by
subtracting the final value and dividing by the initial one.
It can be noticed that Algorithm \ref{algo:main} converges faster than Algorithm \ref{algo:main1}. This fact was confirmed by other simulation results 
performed in various contexts.

\begin{figure}[htbp]
\begin{center}
\includegraphics[height=6cm,width=9cm]{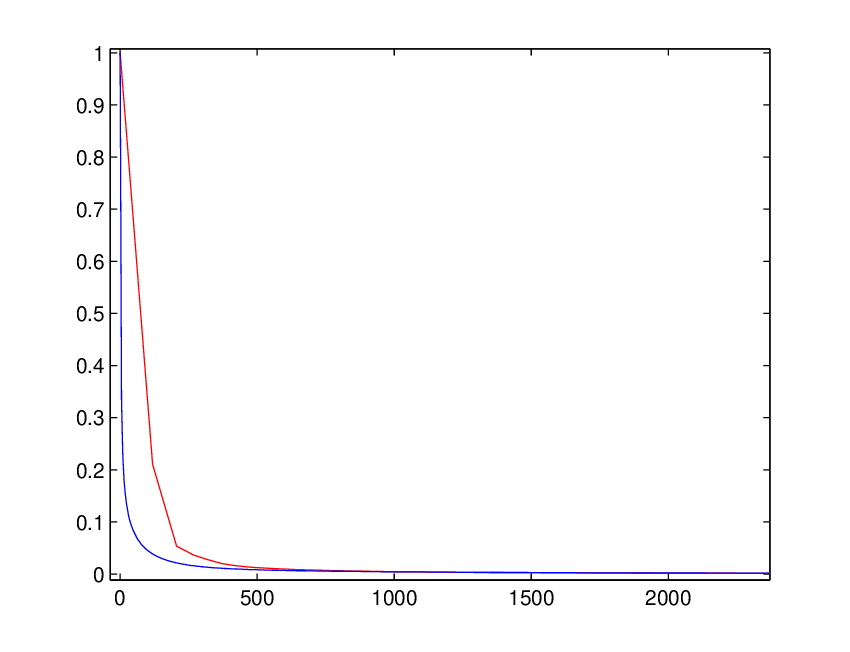}
\end{center}
\caption{Normalized MAP criterion (Algorithm \ref{algo:main1} in red and Algorithm \ref{algo:main} in blue) versus computational time (in seconds) 
(Intel Xeon 4 Core, 3.00 GHz).\label{fig:compalgo}}
\end{figure}

Finally, Fig.~\ref{fig:comptheta} illustrates the influence of the choice
of the parameter $\theta$ when Algorithm \ref{algo:main} is used for a $7 \times 7$ blur and $\alpha_i\equiv 5$.
\begin{figure}[htbp]
\begin{center}
\includegraphics[height=6cm,width=9cm]{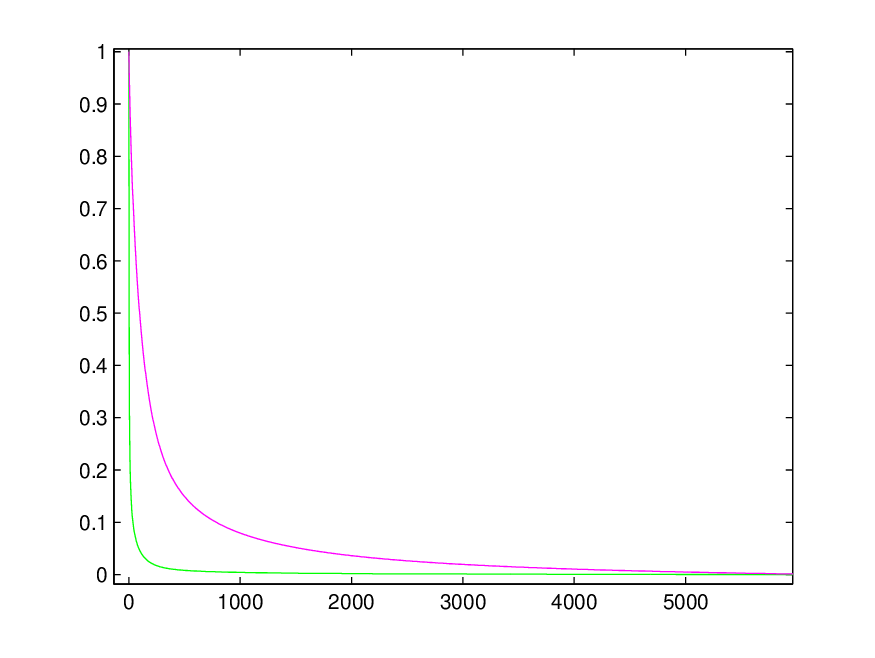}
\end{center}
\caption{Normalized MAP criterion (for $\theta=0.25$ in green and $\theta=10$ in magenta) versus computational time (in seconds) 
(Intel Xeon 4 Core, 3.00 GHz).\label{fig:comptheta}}
\end{figure}
As expected, the larger $\theta$ is, the slower the convergence of the algorithms is.
A trade-off  has therefore to be made: $\theta$ must be chosen large enough to reach a good restoration quality but it should not be too large in order to get a fast convergence.

\subsection{Second example}\label{sec:poisson}
\subsubsection{Model}
In this second scenario, we want to restore an image $\overline{y}\in\RP^N$ 
which is corrupted by a linear operator $T\,:\,\GG \to \GG$, assumed to
be nonnegative-valued and, which is
embedded in (possibly inhomogeneous) Poisson noise. Thus, the observed image $z = (z^{(i)})_{1 \le i \le N} \in \NN^N$ 
is Poisson distributed,
its conditional probability mass function being given by
\begin{equation}
(\forall i \in \{1,\ldots,N\})(\forall \upsilon \in \RP) \qquad
\mu_{Z^{(i)} \mid \overline{U}^{(i)}=\upsilon}(z^{(i)}) = \frac{(\alpha_i \upsilon )^{z^{(i)}}}{z^{(i)}!}\exp\big(-\alpha_i\upsilon \big)
\label{eq:probaPoisson}
\end{equation}
where $(\alpha_i)_{1 \le i \le N} \in \RPP^N$ are scaling parameters.

Consequently, using \eqref{eq:defpsii} and \eqref{eq:probaPoisson}, 
for every $i \in \{1,\ldots,N\}$, 
we have, when $z^{(i)} > 0$,
\begin{equation}
(\forall \upsilon \in \RR)\qquad \psi_i(\upsilon) =
\begin{cases}
\displaystyle \alpha_i \upsilon-z^{(i)}
 + z^{(i)} \ln \Big(\frac{z^{(i)}}{\alpha_i \upsilon}\Big)
& \mbox{if $\upsilon \in \RPP$}\\
+\infty & \mbox{otherwise}
\end{cases}
\label{eq:KL}
\end{equation}
and, when $z^{(i)} = 0$,
\begin{equation*}
(\forall \upsilon \in \RR)\qquad \psi_i(\upsilon) =
\begin{cases}
\displaystyle \alpha_i\upsilon
& \mbox{if $\upsilon \in \RP$}\\
+\infty & \mbox{otherwise.}
\end{cases}
\end{equation*}
As the functions $(\psi_i)_{1\le i \le N}$ are defined up to additive constants,
these constants have been chosen in \eqref{eq:KL} so as to obtain the
 expression of the classical Kullback-Leibler divergence term \cite{Byrne_CL_1993_tip_iter_irabcem}.\\
In this context, provided that $z\neq 0$, Assumption \ref{as:psi} holds with $\delta = 0$
and $\mathbb{I} =$\linebreak$ \menge{i\in \{1,\ldots,N\}}{z^{(i)} > 0}$
since, for all $i\in \mathbb{I}$, 
\begin{align*}
(\forall \upsilon \in \RPP)\qquad \psi_i'(\upsilon) &=\alpha_i - \frac{z^{(i)}}{\upsilon}\\
\psi_i''(\upsilon) & = \frac{z^{(i)}}{\upsilon^{2}} .
\end{align*}
We deduce from \eqref{eq:defuit} that, for every
$i \in \mathbb{I}$,
\begin{equation*}
(\forall \theta \in \RPP)\qquad
\upsilon_i(\theta) = \sqrt{\frac{z^{(i)}}{\theta}}.
\end{equation*}

\begin{remark} \label{re:anscombe}
At this point, it may be interesting to compare the proposed extension with the approach
developed in {\rm \cite{Dupe_FX_2008_ip_proximal_ifdpniusr}}. The use of the Anscombe transform {\rm \cite{Anscombe_F_1948_biometrika_trans_pbnbd}}, 
in {\rm \cite{Dupe_FX_2008_ip_proximal_ifdpniusr}} is actually tantamount to approximating the anti log-likelihood $\psi_i$ of the Poisson distribution  by
\begin{equation}
(\forall \upsilon \in \RR)\qquad
\widetilde{\psi_i}(\upsilon) = 
\begin{cases}
\frac{1}{2} \Big(2 \sqrt{\alpha_i \upsilon+\frac{3}{8}}- 2 \sqrt{z^{(i)}+\frac{3}{8}} \Big)^2
& \mbox{if $\upsilon \in \RP$}\\
\pinf & \mbox{otherwise.}
\end{cases}
\end{equation}
The proposed quadratic extension is illustrated in Fig.~\ref{fig:extensionPois} where  a graphical comparison with the Anscombe approximation  is performed.
\end{remark} 
\begin{figure}[htbp]
\centering
\includegraphics[height=6cm,width=9cm]{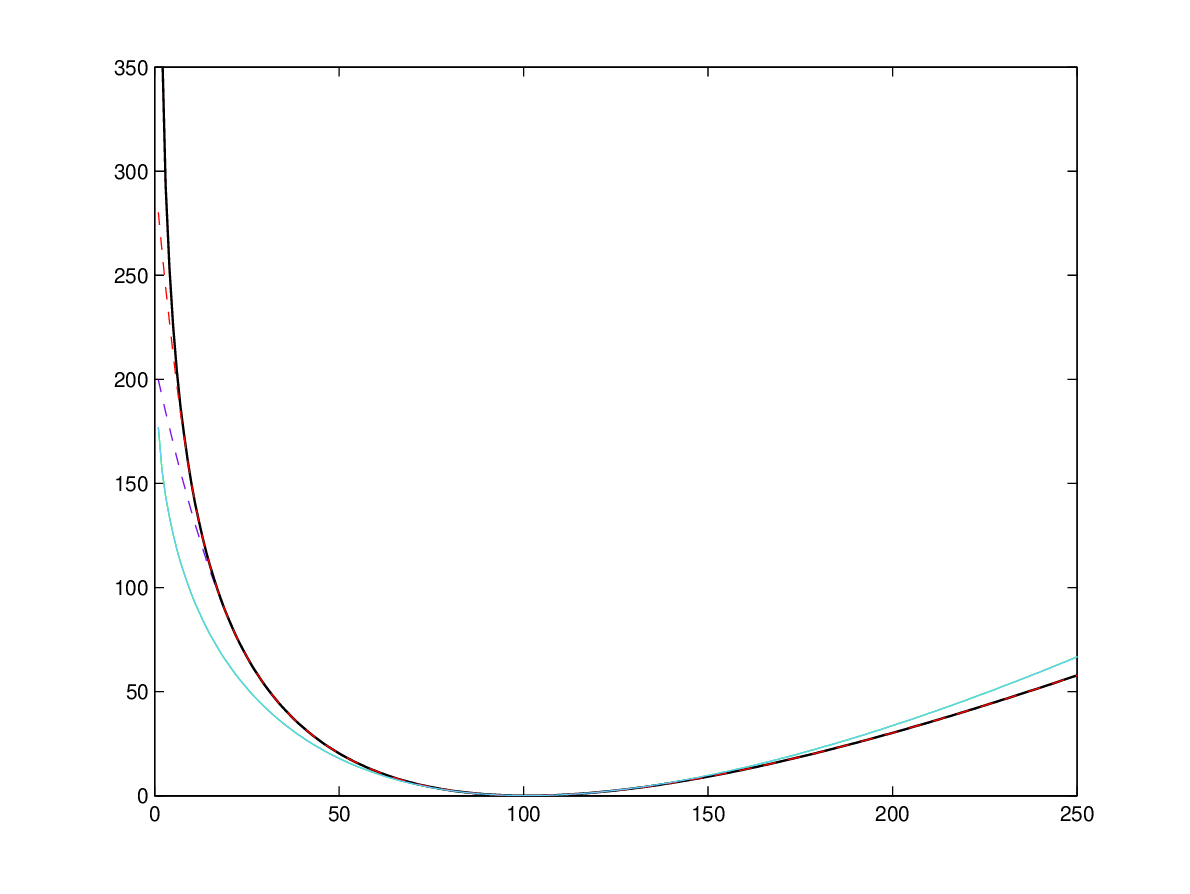} 
\caption{Graph of the function $\psi_i$ (black continuous line) when $\delta = 0$, $\alpha_i = 1$, $z^{(i)} = 100$. Its quadratic extension $\psi_{\theta,i}$  with $\theta= 0.2$ (purple dashed line) and $\theta= 1$ (red dashed line) for $\epsilon(\theta) = 10^{-16}$
 and its Anscombe approximation $\widetilde{\psi}_i$ (cyan continuous line).
\label{fig:extensionPois}}
\end{figure}

\subsubsection{Simulation results}
Here, $T$ is a $5\times 5$ uniform 
blur with $\|T\|=1$. A $256 \times 256$ ($N = 256^2$) medical image $\overline{y}$ shown in Fig. \ref{fig:resPoisson}(a) is degraded by $T$ and corrupted by a Poisson noise following the model described in the previous section for various intensity levels.
The degraded image $z$ is displayed in Fig.~\ref{fig:resPoisson}(b)  when $\alpha_i\equiv 0.01$.

An orthonormal wavelet basis representation has been adopted
using symlets of length $6$ ($\underline{\nu} = \overline{\nu}=1$, $K = N$).
The potential functions $\phi_k$ are taken of the 
same form as in the first example
and, the function $f$ 
is therefore coercive and strictly convex.

The constraint imposed on the solution is given by 
$C = (F^*)^{-1}C^*$ where $C^*$ is defined by
\eqref{eq:constval}.
Since $TC^* = C^*$,
Proposition \ref{p:gt}\ref{p:gtiii}
guarantees that a unique minimizer  $\widetilde{x}_\theta$
of $f+g_\theta+\iota_C$ exists,
which has been computed
with Algorithm \ref{sec:DR(FB)}. 
The algorithm has been initialized by setting $z_0 = P_C z$
and,
we have chosen $\gamma_{m,n}  \equiv 1.99/(\gammat\theta)$,
$\gammat = 60$ and $\lambda_{m,n} \equiv \lambdat_m \equiv 1$.
The number of forward-backward iterations is given by \eqref{eq:Nmchoice}
with $\eta = 10^{-4}$.
Note that the convergence rate could be accelerated by using adaptive step-size methods such as the Armijo-Goldstein search \cite{Tseng_P_2000_jco_modified_fbsmfmmm,Dupe_FX_2008_ip_proximal_ifdpniusr}. However, the computational time of the step-size determination should be taken into account.

To evaluate the performance of our algorithm we use the Signal to Noise Ratio defined in Section \ref{sec:simulSigDep}. Tab. \ref{tab:snr_sebal} shows the values of the $\mathrm{SNR}$
obtained for different values of $\alpha_i$ and $\theta$. As predicted by 
Proposition \ref{p:gt}\ref{p:gtv}, beyond some value of $\theta$,
which is dependent of $\alpha_i$,
 the optimal value is found. We also compare our results with those provided by two different approaches. The first one is the regularized Expectation Maximization (EM) approach
(also sometimes called SMART) \cite{Byrne_CL_1993_tip_iter_irabcem,Lange_K_1987_tmi_theoretical_atsosmlafeatt}  where the Poisson anti-likelihood penalized by
a term proportional to the Kullback-Leibler divergence between the desired solution
and a reference image is minimized. Its weighting factor has been adjusted manually so as to maximize the $\mathrm{SNR}$ and,
the reference image is a constant image whose pixel values has been set to the mean value of the degraded image. The other approach is the method based on the Anscombe transform proposed in \cite{Dupe_FX_2008_ip_proximal_ifdpniusr} and discussed in Remark \ref{re:anscombe}.
For fair comparisons, the method here employs the same orthonormal wavelet representation, the same 
functions $(\phi_k)_{1\le k \le K}$ as ours and the same constraint set
$C$.
 It can be observed that the approach we propose gives good results. However, for high intensity levels ($\alpha_i \ge 0.1$), the method based on the Anscombe transform performs equally well in terms of SNR.
The restored images are shown in Fig. \ref{fig:resPoisson}, when $\alpha_i \equiv 0.01$ and $\theta \equiv 0.001$ after 3000 iterations. In spite of an important degradation of the original image, it can be seen that our approach is able to recover the main features in the image. 
It can also be noticed that the image restored by the two methods exhibit different visual characteristics.

\begin{table}[htbp]
\centering
\begin{tabular}{| p{0.5cm}||c|| c||c|c|c|c|c|}
\hline
& Regularized &Anscombe & \multicolumn{5}{c|}{Quadratic extension}\\
\cline{4-8}
$\alpha_i$&EM & &$\theta=0.001$ & $\theta=0.005$ & $\theta=0.1$ & $\theta=1$ & $\theta=5$\\
\hline
\hline
$  0.01$ & 6.47 & 8.24 & \textbf{9.75} & 9.75 & 9.75 & 9.75 & 9.75\\
\hline
$0.05$ & 9.01 & 11.5 & 11.7 & \textbf{11.9} & 11.9 & 11.9 & 11.9 \\
\hline
$0.1$ & 10.1 & 12.4 & 12.0 & \textbf{12.5}  & 12.5 & 12.5 & 12.5\\
\hline
$1$ & 13.8 & \textbf{15.1}   & 0    & 10.1 & 13.7 & \textbf{15.1} & 15.1\\
\hline
\end{tabular}
\caption{$\mathrm{SNR}$ for the medical image.}
\label{tab:snr_sebal}
\end{table}

\begin{figure}[htbp]
\begin{center}
\begin{tabular}{cc}
\multicolumn{2}{c}{\includegraphics[height=5cm]{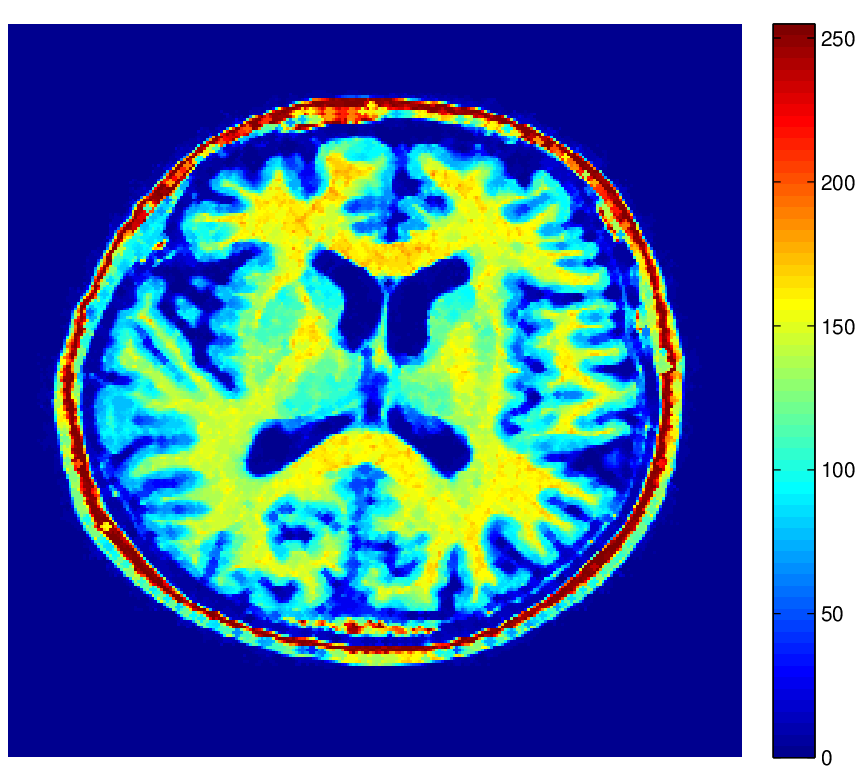}}\\
\multicolumn{2}{c}{(a)}\\
\includegraphics[width=5cm]{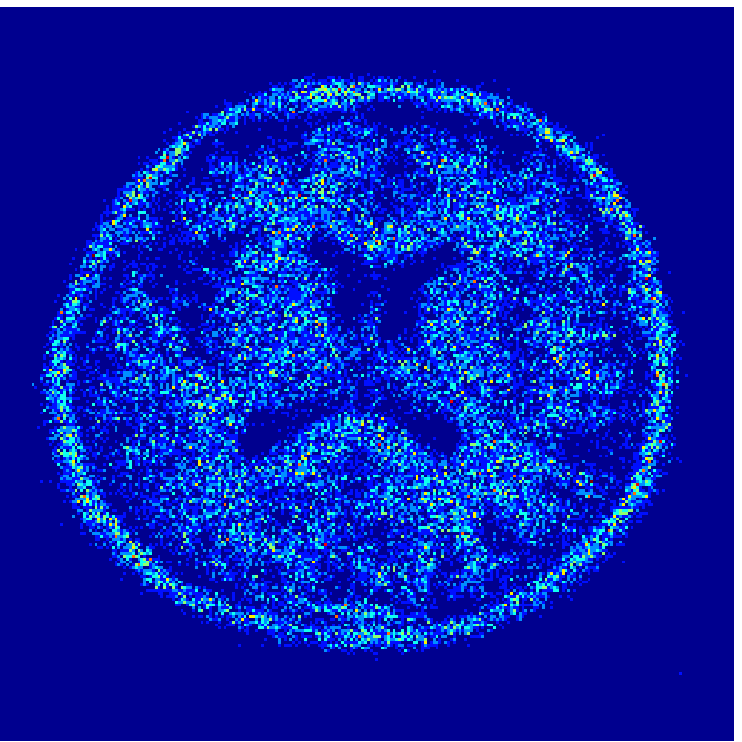} & \includegraphics[width=5cm]{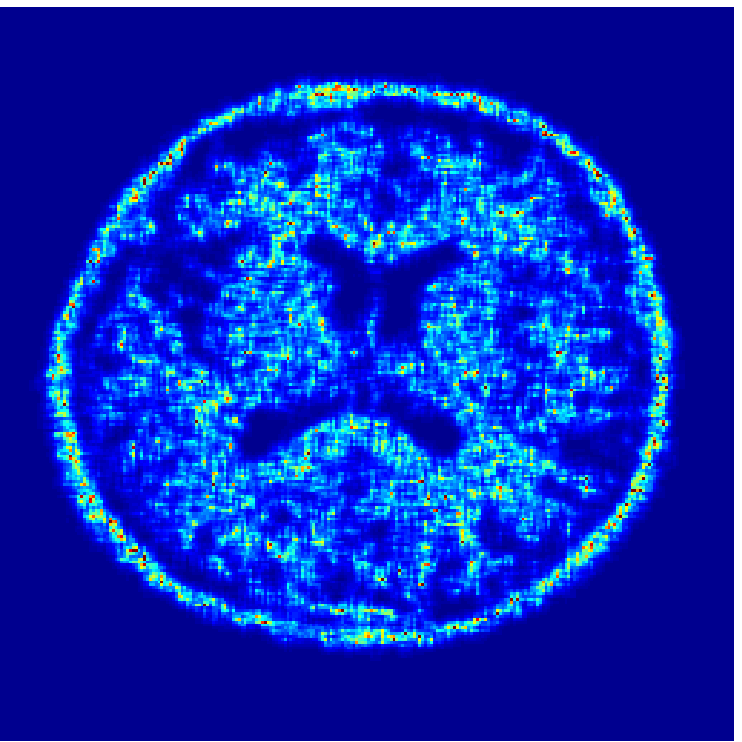} \\
(b) & (c) \\
\includegraphics[width=5cm]{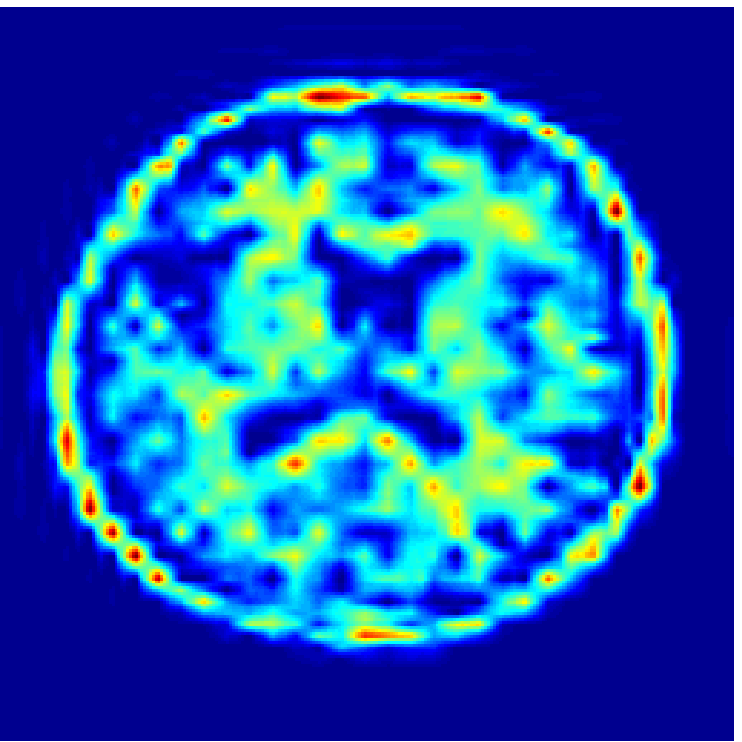} & \includegraphics[width=5cm]{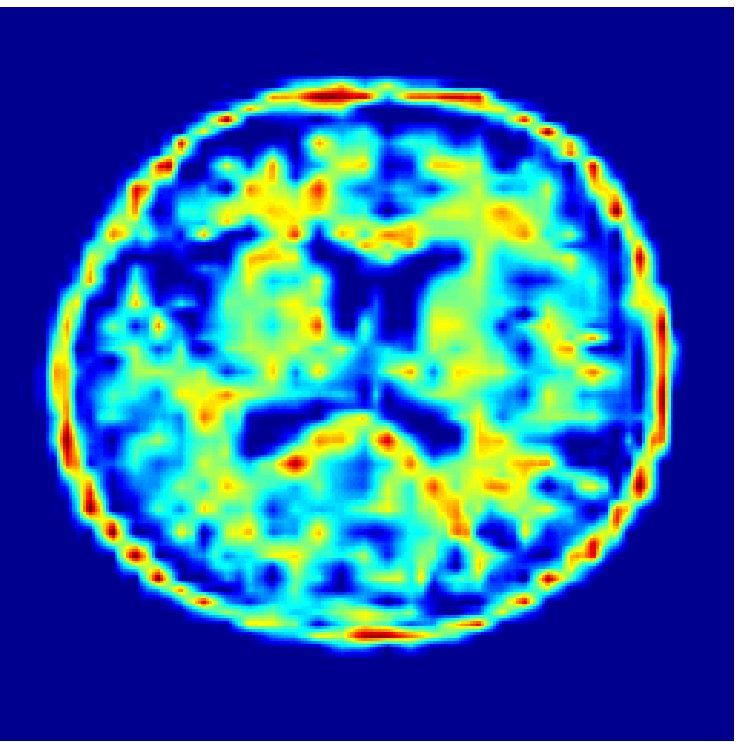}\\
(d) & (e)\\
\end{tabular}
\end{center}
\caption{Results on the medical image. (a) Original, (b) degraded, (c) restored with EM, (d) restored with Anscombe transform and (e) restored with quadratic extension.
\label{fig:resPoisson}}
\end{figure}

\section{Conclusion}
Two main problems have been addressed in this paper.

The first one concerns the minimization on a convex set $C$ of a sum of 
two functions, one of which $g$ being smooth while the other may be nonsmooth. Such a constrained minimization has been performed by combining forward-backward and Douglas-Rachford iterations. Various combinations of these algorithms can be envisaged and the study we made tends to show that Algorithm~\ref{algo:main} is a 
good choice.
It can be noticed that adding a constraint on the solution for a restoration problem was shown to be useful in another work \cite{Pustelnik_N_2008_peusipco_constrained_fbairp}, where it appeared that the visual quality of the restored image can be much improved w.r.t. the unconstrained case, when both restoration approaches are applicable.

The second point concerns the quadratic lower approximation technique we have proposed. 
This method offers a means of applying the proposed algorithms in cases when $g$ is differentiable on $C$ but the gradient of $g$ is not necessary Lipschitz continuous on $C$. By quadratically extending $g$, 
the proposed constrained minimization algorithms can be used.
This extension depends on a parameter $\theta$ which controls the precision (closeness to the solution of the original minimization problem) and 
the convergence speed of the algorithm.
As illustrated by the simulations, the choice 
of this parameter should result from a trade-off. The numerical results have also shown the efficiency of the proposed methods in deconvolution problems involving a signal-dependent Gaussian noise or a
Poisson noise.

Finally, it may be interesting to note that nested iterative algorithms similar to those developed in this paper can be used to solve
$\min_{x\in\HH} f + g + h$ where $\HH$ is a real separable Hilbert space,
$f$, $g$ 
and $h$ are functions in $\Gamma_0(\HH)$ and $g$ is $\beta$-Lipschitz differentiable.

\begin{appendix}

\section{Study of Example \ref{ex:ce1}} \label{ap:exce1}
Let $p=\prox_{f} x$ and $q = \prox_{f+\iota_C} x$ where
$x \in \HH$.
Let $g$ be the convex function defined by $(\forall y \in \HH)$ 
$g(y) = \frac{1}{2}\left\|y-x\right\|^2 + \frac{1}{2}y^{\top}\Lambda y$. Consequently, $p=\big(I+\Lambda\big)^{-1}x$ is the minimizer of $g$ on $\HH$, whereas $q$ is the minimizer of $g$ on $C$. Thus, we can write $(\forall y \in \HH)$ $g(y) = \tilde{g}(y)+h_x$ where $\tilde{g}(y) = \frac{1}{2}(y-p)^{\top}(I+\Lambda)(y-p)$ and $h_x$ is a function of $x$. Then, $q$ also minimizes $\tilde{g}$ on $C$.\\
In the example, we have chosen $x=2(\Lambda_{1,2},1+\Lambda_{2,2})^{\top}$, which yields $ p = (0,2)^{\top}$ and $P_C(p) = (0,1)^{\top}$.\\
Let $\tilde{q} = (\pi,1)^{\top}$. 
To show that $q = \tilde{q}$, we have check that $\tilde{q}$ minimizes $\tilde{g}$
on $C$. A necessary and sufficient condition for the latter property
to be satisfied \cite[p. 293, Theorem 1.1.1]{Hiriart_Urruty_1996_book_convex_amaIf} is that 
\begin{equation*}
(\forall y \in C)\qquad \big(\nabla\tilde{g}(\tilde{q})\big)^\top(y-\tilde{q}) \ge 0
\end{equation*}
where 
$\nabla\tilde{g}(\tilde{q}) = (I+\Lambda)(\tilde{q}-p)$ is
the gradient of $\tilde{g}$ at $\tilde{q}$.
This is equivalent to prove that
\begin{equation}
(\forall (y^{(1)},y^{(2)})^\top \in C)\qquad
(2\pi-\Lambda_{1,2})(y^{(1)}-\pi) + (\Lambda_{1,2} \pi - \Lambda_{2,2}-1)(y^{(2)}-1) \geq 0.
\label{eq:minequac}
\end{equation}
Three cases must be considered: 
\begin{itemize}
\item when $\Lambda_{1,2}<-2$, $(y^{(1)},y^{(2)})^\top \in C$ $\Rightarrow$
$y^{(1)} \ge -1 = \pi$ and $y^{(2)} \le 1$. In addition, we have
$2\pi-\Lambda_{1,2} = -2 -\Lambda_{1,2} > 0$ and 
$\Lambda_{2,2}-\Lambda_{1,2}^{2}\geq 0$ $\Rightarrow$ 
$\Lambda_{1,2} \pi - \Lambda_{2,2}-1 \le -\Lambda_{1,2}^{2}-\Lambda_{1,2}-1 < 0$.
So, \eqref{eq:minequac} holds.
\item When $\Lambda_{1,2}>2$, similar arguments hold.
\item When $\Lambda_{1,2}\in[-2,2]$, $2\pi-\Lambda_{1,2} = 0$
and $\Lambda_{1,2} \pi - \Lambda_{2,2}-1
= \frac{\Lambda_{1,2}^2}{2} - \Lambda_{2,2}-1
\le -\frac{\Lambda_{1,2}^2}{2}-1 \le 0$, which shows that
\eqref{eq:minequac} is satisfied.
\end{itemize}
This leads to the conclusion of Example \ref{ex:ce1}.

\section{Study of Example \ref{ex:ce2}}\label{ap:exce2}
Let $f$ be the function defined in Example~\ref{ex:ce2}.
Defining the rotation matrix $R=\frac{1}{\sqrt 2}\begin{pmatrix} 1 & -1 \\ 1 & 1 \end{pmatrix}$, this function can be expressed as
\begin{equation*}
(\forall x\in \RR^{2}) \qquad f(x) = \tilde{f}(Rx)
\end{equation*}
where $\tilde{f}(x) = \frac{1}{2} x^\top \Lambda x$ with
\begin{equation*}
\Lambda = \begin{pmatrix}
1 & \Lambda_{1,2}\\
\Lambda_{1,2} & 1
\end{pmatrix}.
\end{equation*}
In addition, 
\begin{equation*}
C = \{x\in \RR^{2}\;\mid \;Rx\in [-1,1]^2\} = R^\top [-1,1]^2.
\end{equation*}
It can be noticed that $[-1,1]^2$ is the separable convex set considered
in Example \ref{ex:ce1} whereas $\tilde{f}$ appears as a particular case in the class of quadratic
functions considered in this example (by setting $\Lambda_{2,2} = 1$).\\
Thus, the proximity operator of $f$ is
\begin{align*}
(\forall x \in \HH)\qquad
\prox_f x = &\arg\min_{y\in \HH} \frac{1}{2} \|x-y\|^2 + f(y)\nonumber\\
= & \arg\min_{y\in \HH} \frac{1}{2} \|Rx-Ry\|^2 + \tilde{f}(Ry)
=  R^\top \prox_{\tilde{f}}(Rx).
\end{align*}
and $P_C(\prox_f x) = R^\top P_{[-1,1]^2}(R\prox_f x)
= R^\top P_{[-1,1]^2}\big(\prox_{\tilde{f}}(R x)\big)$.
Similarly, we have
\begin{equation*}
(\forall x \in \HH)\qquad
\prox_{f+\iota_C} x = R^\top \prox_{\tilde{f}+\iota_{[-1,1]^2}}(Rx).
\end{equation*}
So, if $x = 2R^\top (\Lambda_{1,2},2)^\top = 
\sqrt{2}(2+\Lambda_{1,2},2-\Lambda_{1,2})^\top$, we deduce from
Example \ref{ex:ce1} that
$P_C(\prox_f x) = \frac{1}{\sqrt{2}}(1,-1)^\top$
and $\prox_{f+\iota_C} x = \frac{1}{\sqrt{2}} (1+\pi+1,1-\pi)^\top$,
where the expression of $\pi$ is given by \eqref{eq:defpi}.
It can be concluded that $P_C(\prox_f x)\neq \prox_{f+\iota_C} x$.

\end{appendix}

\end{document}